\documentclass{article}
\usepackage{amsmath}
\usepackage{amssymb}
\usepackage{amsfonts}
\usepackage{amsthm}
\theoremstyle{plain}
\newtheorem{theorem}{Theorem}[section]
\newtheorem{lemma}{Lemma}[section]
\newtheorem{proposition}{Proposition}[section]
\newtheorem{corollary}{Corollary}[section]
\theoremstyle{definition}
\newtheorem{definition}{Definition}[section]
\newtheorem{example}{Example}[section]
\theoremstyle{remark}
\newtheorem{remark}{Remark}[section]
\newcommand\latop[2]{{#1\atop#2}}
\newcommand\oooo[4]{\begin{picture}(87,12)
\put(3,3){\makebox(0,0){$\circ$}}\put(5,3){\line(1,0){18}}
\put(25,3){\makebox(0,0){$\circ$}}\put(27,3){\line(1,0){6}}
\put(44,3){\makebox(0,0){\dots}}\put(55,3){\line(1,0){6}}
\put(63,3){\makebox(0,0){$\circ$}}\put(65,3){\line(1,0){18}}
\put(85,3){\makebox(0,0){$\circ$}}
\put(3,10){\makebox(0,0){\scriptsize $#1$}}
\put(25,10){\makebox(0,0){\scriptsize $#2$}}
\put(63,10){\makebox(0,0){\scriptsize $#3$}}
\put(85,10){\makebox(0,0){\scriptsize $#4$}}
\end{picture}}
\newcommand\tab{\begin{picture}(69,13)
\put(1,13){\line(0,-1){12}}\put(7,13){\line(0,-1){12}}
\put(1,1){\line(1,0){34}}\put(1,7){\line(1,0){68}}\put(1,13){\line(1,0){68}}
\put(18,10){\makebox(0,0){\dots}}\put(18,4){\makebox(0,0){\dots}}
\put(29,13){\line(0,-1){12}}\put(35,13){\line(0,-1){12}}
\put(41,13){\line(0,-1){6}}\put(63,13){\line(0,-1){6}}\put(69,13){\line(0,-1){6}}
\put(52,10){\makebox(0,0){\dots}}
\end{picture}}
\newcommand{\one}{\begin{picture}(7,7)
\put(1,7){\line(0,-1){6}}\put(1,7){\line(1,0){6}}\put(1,1){\line(1,0){6}}\put(7,7){\line(0,-1){6}}
\end{picture}}
\newcommand{\secsym}{\begin{picture}(13,7)
\put(1,7){\line(0,-1){6}}\put(7,7){\line(0,-1){6}}\put(13,7){\line(0,-1){6}}
\put(1,7){\line(1,0){12}}\put(1,1){\line(1,0){12}}
\end{picture}}
\newcommand{\secskewsym}{\begin{picture}(7,13)
\put(1,13){\line(0,-1){12}}\put(7,13){\line(0,-1){12}}
\put(1,13){\line(1,0){6}}\put(1,7){\line(1,0){6}}\put(1,1){\line(1,0){6}}
\end{picture}}
\newcommand{\thirdsym}{\begin{picture}(19,7)
\put(1,7){\line(0,-1){6}}\put(7,7){\line(0,-1){6}}
\put(13,7){\line(0,-1){6}}\put(19,7){\line(0,-1){6}}
\put(1,7){\line(1,0){18}}\put(1,1){\line(1,0){18}}
\end{picture}}
\newcommand{\corner}{\begin{picture}(13,13)
\put(1,13){\line(0,-1){12}}\put(7,13){\line(0,-1){12}}\put(13,13){\line(0,-1){6}}
\put(1,13){\line(1,0){12}}\put(1,7){\line(1,0){12}}\put(1,1){\line(1,0){6}}
\end{picture}}
\newcommand{\thirdskewsym}{\begin{picture}(7,19)
\put(1,19){\line(0,-1){18}}\put(7,19){\line(0,-1){18}}
\put(1,19){\line(1,0){6}}\put(1,13){\line(1,0){6}}
\put(1,7){\line(1,0){6}}\put(1,1){\line(1,0){6}}
\end{picture}}
\DeclareMathOperator{\End}{End}

\DeclareMathOperator{\depth}{depth}
\DeclareMathOperator{\sgn}{sgn}

\begin{document}

\title{Symmetries of CR sub-Laplacian}

%\classification{58J70}
%\keywords{CR geometry, Invariant operators, Ambient construction, symmetry operators}

\author{Zuzana Vlas\' akov\' a}

\maketitle

\begin{abstract}
We define a CR structure on a distinguished hyperplane in $\mathbb{C}^{n+1}$ and the CR sub-Laplacian on this CR manifold. We also define symmetries of the CR sub-Laplacian in general and for this special case construct all of them using the ambient construction. Then we investigate the algebra structure of the symmetry algebra of the sub-Laplacian. For this purpose we derive the decomposition of $S^{k}_{0}\mathfrak{sl}(V)$ under the action of $SL(V)$.
\end{abstract}

\section{Introduction}

Invariant differential operators have a long story of importance and this is particularly the case for operators of Laplace type. The conformally invariant Laplacian is the basic example in conformal geometry. A family of higher order generalizations of the conformal Laplacian with principal part a power of the Lalacian was constructed in \cite{powlap}. In CR geometry, the CR invariant sub-Laplacian of Jerison-Lee (\cite{sublap}) plays a role analogous to that of the conformal Laplacian. In \cite{powsublap} generalizations of the Jerison-Lee sub-Laplacian are defined, which are the CR analogues of the 'conformally invariant powers of the Laplacian'.\par
This work was inspired by the article \cite{symlap} by M. Eastwood, and the diploma thesis \cite{diplomka} by V{\'\i}t Tu\v cek. The aim was to characterize the vector space of all symmetries of the CR sub-Laplacian. In the paper \cite{symlap} author identifies the symmetry algebra of the Laplacian on the Euclidean space as an explicit quotient of the universal enveloping algebra of the Lie algebra of conformal motions and constructs analogues of these symmetries on a general conformal manifold.\par
The space of smooth first order linear differential operators on $\mathbb{R}^{n}$ that preserve harmonic functions is closed under Lie bracket. For $n\geq3$, it is finite-dimensional (of dimension $(n^{2}+3n+4)/2$). Its commutator algebra is isomorphic to $\mathfrak{so}(n+1,1)$, the Lie algebra of conformal motions of $\mathbb{R}^{n}$. Second order symmetries of the Laplacian on $\mathbb{R}^{3}$ were classified by Boyer, Kalnis, and Miller in \cite{sep}. Commuting pairs of second order symmetries, as observed by Winternitz and Fri\v s in \cite{commsym}, correspond to separation of variables for the Laplacian. This leads to classical coordinate systems and special functions, see \cite{sep} and \cite{miller}.\par
General symmetries of the Laplacian on $\mathbb{R}^{n}$ give rise to an algebra, filtered by degree. For $n\geq3$, the filtering subspaces are finite-dimensional and closely related to the space of conformal Killing tensors. The main result of \cite{symlap} is an explicit algebraic description of this symmetry algebra. The motivation for \cite{symlap} has come from physics, especially the theory of higher spin fields and their symmetries.\par
In section 2 we will define a CR structure on a distinguished hyperplane in $\mathbb{C}^{n+1}$ (it is in fact the big cell in the homogeneous model of CR geometry viewed as a parabolic geometry together with the very flat Weyl structure on it, see \cite{parabook}) and the CR sub-Laplacian on this CR manifold. We also define symmetries of the CR sub-Laplacian in general. In section 3 we give a classification of symmetries via properties of their symbol. In section 4 we introduce the ambient construction and use it to construct symmetries of the sub-Laplacian, which enables us to prove existence of symmetries and to give a characterization of vector space of symmetries as an $SL(n+2,\mathbb{C})$-module. In section 6 we establish the algebra structure of the symmetry algebra using the commutant of the action of $SL(n+2,\mathbb{C})$ on $S^{k}_{0}\mathfrak{sl}(n+2,\mathbb{C})$ computed in section 5.

\section{Basic definitions}

A \emph{Levi-nondegenerate CR structure of hypersurface type} on a $(2n+1)$-dimensional manifold $M$ is a subbundle $HM\subset TM$ of real codimension 1 endowed with an integrable complex structure, s.t. the \emph{Levi bracket} $\mathcal{L}:HM\times HM\to TM/HM=:QM$ given by $\mathcal{L}(X,Y)=p([X,Y])$, where $p$ is the canonical projection, is nondegenerate. The nondegeneracy of the Levi bracket is equivalent to the fact that $HM$ induces a contact structure on $M$. The Levi bracket may be thought of as minus twice the imaginary part of some Hermitian inner product, and signature of this inner product is called the \emph{signature} of $M$. None of what follows will depend on the signature. Since we will only consider the case of Levi-nondegenerate CR structures of hypersurface type, we will for brevity call them CR structures.\par
On each CR manifold $M$ we have an $n$-dimensional complex vector bundle $HM^{(1,0)}\subset TM\otimes\mathbb{C}$ and its conjugate $HM^{(0,1)}\subset TM\otimes\mathbb{C}$. Define $\Lambda^{(1,0)}\subset T^{*}M\otimes\mathbb{C}$ by $\Lambda^{(1,0)}=(HM^{(0,1)})^{\bot}$. The \emph{canonical bundle} $\mathcal{K}:=\Lambda^{n+1}(\Lambda^{(1,0)})$ is a complex line bundle on $M$. We will assume that $\mathcal{K}$ admits an $(n+2)$nd root and we fix a bundle denoted by $\mathcal{E}(1,0)$, which is a $-1/(n+2)$th power of $\mathcal{K}$. The bundle $\mathcal{E}(w_{1},w_{2}):=(\mathcal{E}(1,0))^{w_{1}}\otimes(\overline{\mathcal{E}(1,0)})^{w_{2}}$ of $(w_{1},w_{2})$-densities is defined for $w_{1},w_{2}\in\mathbb{C}$ satisfying $w_{1}-w_{2}\in\mathbb{Z}$. If $\mathcal{E}(1,0)\setminus\{0\}$ is viewed as $\mathbb{C}^{\times}$-principal bundle, then $\mathbb{E}(w_{1},w_{2})$ is the bundle induced by the representation $\lambda\mapsto\lambda^{-w_{1}}\overline{\lambda}^{-w_{2}}$. The bundle $QM\otimes\mathbb{C}$ can be identified with $\mathcal{E}(1,1)$.\par
Let's consider $\mathbb{C}^{n+1}(z^{1},\dots,z^{n},z^{\infty})$ with Hermitian metric of the form
\begin{displaymath}
\left(\begin{matrix}
g_{\bar{a}b}&0\\
0&0
\end{matrix}\right)	
\end{displaymath}
where $g_{\bar{a}b}$ is of signature $(p,q)$ with $p+q=n$ and consider a submanifold $M\subset\mathbb{C}^{n+1}$ given by
\begin{displaymath}
\sum_{a=1}^{n}z^{a}z_{a}+z^{\infty}+\bar{z}^{\infty}=0	
\end{displaymath}
On the manifold $M$ we shall define a CR structure. In coordinates the submanifold $M$ looks like 
\begin{displaymath}
M=\{(z^{1},\dots,z^{n},-\sum_{a=1}^{n}\frac{z^{a}z_{a}}{2}+i\sigma)\in\mathbb{C}^{n+1}\}
\end{displaymath}
where we put $z^{\infty}=\rho+i\sigma$. In terms of coordinates on $\mathbb{C}^{n+1}$, the coordinate vector fields on $M$ look like $\partial_{z^{a}}-\frac{z_{a}}{2}\partial_{z^{\infty}}-\frac{z_{a}}{2}\partial_{\bar{z}^{\infty}}$, their conjugates, and $\partial_{\sigma}=i(\partial_{z^{\infty}}-\partial_{\bar{z}^{\infty}})$. The contact subbundle $HM^{(1,0)}\subset (T\mathbb{C}^{n+1})^{(1,0)}$ has basis $\{\partial_{a}=\partial_{z^{a}}-z_{a}\partial_{z^{\infty}},a=1,\dots,n\}$, since it has to be formed by complex linear combinations of coordinate vector fields holomorphic as vector fields on $\mathbb{C}^{n+1}$. The only nontrivial commutator is
\begin{displaymath}
[\partial_{\bar{a}},\partial_{b}]=ig_{\bar{a}b}\partial_{\sigma}	
\end{displaymath}
We will work mostly with vector fields $\partial_{\bar{a}}, \partial_{b}$, the coordinate vector fields will be without use.\par
Our manifold $M$ is in fact a big cell (see \cite{parabook}) in the corresponding Hermitian quadric (depending on the signature) in complex projective space, which is a homogeneous model (the quadric) of hypersurface-type nondegenerate CR structure of corresponding signature. Such a structure is a special case of parabolic structure, so we can use the theory of parabolic geometries. In particular, we can use the notion of Weyl structure (\cite{parabook}, \cite{Weylstr}), concretely the very flat Weyl structure on the big cell of the homogeneous model of such geometry. From now on we will work with this concrete Weyl structure. So we can identify all densities with functions and the Weyl derivative on these bundles with ordinary derivatives, i.e. $\nabla_{a}\equiv\partial_{a}$ and so on.The advantage of this Weyl structure is that in the formulae for all operators we will work with the curvature terms vanish.
\begin{definition}
The \emph{CR sub-Laplacian} $\Delta:\mathcal{E}(w_{1},w_{2})\to\mathcal{E}(w_{1}-1,w_{2}-1)$ on $M$ for $n+w_{1}+w_{2}=0$ is given by
\begin{displaymath}
\Delta(f):=g^{a\bar{b}}(\partial_{a}\partial_{\bar{b}}+\partial_{\bar{b}}\partial_{a})(f)+i\frac{w_{1}-w_{2}}{2}\partial_{\sigma}(f)	
\end{displaymath}
where we implicitly use the Einstein summation convention.
\end{definition}
\begin{definition}
\begin{itemize}
\item [a)]A \emph{symmetry} of $\Delta$ is a linear differential operator $\mathcal{D}$, s.t. there exists a differential operator $\delta$ satisfying $\delta\Delta=\Delta D$.
\item [b)]A symmetry of $\Delta$ is called \emph{trivial}, if it is of the form $P\Delta$ for some linear differential operator $\Delta$. Two symmetries $\mathcal{D}_{1}$ and $\mathcal{D}_{2}$ are called \emph{equivalent}, if their difference is a trivial symmetry.
\end{itemize}
\end{definition}
Since the trivial symmetries are not very interesting, we will only consider the vector space of symmetries modulo the equivalence relation that symmetries differ by a trivial symmetry.

\section{Properties of the symbol}

The algebra $\mathcal{A}$ of symmetries is naturally filtered by order $d$ of  operators. For our purposes let us introduce a finer filtration.
\begin{definition}
We say that the term $\partial_{a_{1}}\dots\partial_{a_{k}}\partial_{\bar{b}_{1}}\dots\partial_{\bar{b}_{l}}\partial_{\sigma}\dots\partial_{\sigma}$ (with $\partial_{\sigma}$ being $m$-times) as being of degree $(d,s)$, if $k+l+m=d$ and $s=\min(k,l)$.\\
We put $(d,s)\leq(d',s')$ if $d\leq d'$ or $d=d'$ and $s\geq s'$. Here $2s\leq d$.\\
Then we define
\begin{displaymath}
	\mathcal{A}^{(d,s)}=\bigoplus_{(d',s')\leq(d,s)}\mathcal{A}_{(d',s')}
\end{displaymath}
\end{definition}
Before stating the theorem, we shall emphasize that the symmetries, since acting on \emph{complex}-valued functions (densities), form naturally a \emph{complex} vector space. Therefore all representations we will work with will be \emph{complex} (may be viewed as representations of $SL(n+2,\mathbb{C})$) and by a BGG operator we will mean its complex-linear extension to the complexification of corresponding bundle.
\begin{theorem}\label{thmsymb}
Every symmetry is equivalent to some of the form
\begin{displaymath}
P=\sum_{\latop{k+l\leq d}{s=\min(k,l)}}V^{a_{1}\dots a_{k}\bar{b}_{1}\dots\bar{b}_{l}\sigma\dots\sigma}\partial_{a_{1}}\dots\partial_{a_{k}}\partial_{\bar{b}_{1}}\dots\partial_{\bar{b}_{l}}
\partial_{\sigma}\dots\partial_{\sigma}+LDTS
\end{displaymath}
with each $V^{a_{1}\dots a_{k}\bar{b}_{1}\dots\bar{b}_{l}\sigma\dots\sigma}$ having exactly $d$ indices and
\begin{itemize}
  \item[(1)] is maximally symmetric and trace-free
  \item[(2)] $V^{a_{1}\dots a_{s}\bar{b}_{1}\dots\bar{b}_{s}\sigma\dots\sigma}$ is a solution of the first BGG operator corresponding to $\quad\oooo{d-2s}{s}{s}{d-2s}\quad$ and
  \begin{gather*}
  V^{a_{1}\dots a_{k+s}\bar{b}_{1}\dots\bar{b}_{s}\sigma\dots\sigma}=\frac{i^{k}}{k!}\partial^{(a_{1}}\dots\partial^{a_{k}}V^{a_{k+1}\dots a_{k+s})\bar{b}_{1}\dots\bar{b}_{s}\sigma\dots\sigma}\\
  V^{a_{1}\dots a_{s}\bar{b}_{1}\dots\bar{b}_{k+s}\sigma\dots\sigma}=  \frac{(-i)^{k}}{k!}\partial^{(\bar{b}_{1}}\dots\partial^{\bar{b}_{k}}V^{\bar{b}_{k+1}\dots\bar{b}_{k+s})a_{1}\dots a_{s}\sigma\dots\sigma}
  \end{gather*}
\end{itemize}  
\end{theorem}
\begin{proof}
\begin{itemize}
  \item[(1)] Every symmetry of order $d$ can be written as
\begin{displaymath}
P=\sum_{k+l\leq d}V^{a_{1}\dots a_{k}\bar{b}_{1}\dots\bar{b}_{l}\sigma\dots\sigma}\partial_{a_{1}}\dots\partial_{a_{k}}\partial_{\bar{b}_{1}}\dots\partial_{\bar{b}_{l}}
\partial_{\sigma}\dots\partial_{\sigma}+LOTS
\end{displaymath}
where each $V^{a_{1}\dots a_{k}\bar{b}_{1}\dots\bar{b}_{l}\sigma\dots\sigma}$ is totally symmetric (because the commutator of two derivatives gives a term of lower order) and has exactly $d$ indices. To prove that it may be considered trace-free, let's assume without loss of generality that it has a trace in $a_{k}\bar{b}_{l}$, hence it has a summand of the form $g^{a_{k}\bar{b}_{l}}W^{a_{1}\dots a_{k-1}\bar{b}_{1}\dots\bar{b}_{l-1}\sigma\dots\sigma}$. We want to show that using the equivalence relation we can leave out this term. We commute the derivatives $\partial_{a_{k}}\partial_{\bar{b}_{l}}$ to the end (it does not effect the symbol) to get
\begin{displaymath}
g^{a_{k}\bar{b}_{l}}W^{a_{1}\dots a_{k-1}\bar{b}_{1}\dots\bar{b}_{l-1}\sigma\dots\sigma}\partial_{a_{1}}\dots\partial_{a_{k-1}}\partial_{\bar{b}_{1}}\dots\partial_{\bar{b}_{l-1}}\partial_{\sigma}\dots\partial_{\sigma}\partial_{a_{k}}\partial_{\bar{b}_{l}}+LOTS	
\end{displaymath}
Replacing $\partial_{a_{k}}\partial_{\bar{b}_{l}}$ by $\frac{1}{2}(\partial_{a_{k}}\partial_{\bar{b}_{l}}+\partial_{\bar{b}_{l}}\partial_{a_{k}})$ (the commutator is of lower order), we get the leading term of $W^{a_{1}\dots a_{k-1}\bar{b}_{1}\dots\bar{b}_{l-1}\sigma\dots\sigma}\partial_{a_{1}}\dots\partial_{a_{k-1}}\partial_{\bar{b}_{1}}\dots\partial_{\bar{b}_{l-1}}\partial_{\sigma}\dots\partial_{\sigma}$ composed with the sub-Laplacian. Subtracting this composition, we get an operator of lower order. So the trace part of $V^{a_{1}\dots a_{k}\bar{b}_{1}\dots\bar{b}_{l}\sigma\dots\sigma}$ can be left out using the equivalence, hence we may consider it being trace-free.
  \item[(2)] If we commute $P$ with the sub-Laplacian, the leading term of the commutator consists of terms of two types. Every $V^{a_{1}\dots a_{k}\bar{b}_{1}\dots\bar{b}_{l}\sigma\dots\sigma}$ gives rise to a term \begin{gather*}
\partial^{(a}V^{a_{1}\dots a_{k})\bar{b}_{1}\dots\bar{b}_{l}\sigma\dots\sigma}\partial_{a}\partial_{a_{1}}\dots\partial_{a_{k}}\partial_{\bar{b}_{1}}\dots\partial_{\bar{b}_{l}}\partial_{\sigma}\dots\partial_{\sigma}+\\
+\partial^{(\bar{b}}V^{\bar{b}_{1}\dots\bar{b}_{l})a_{1}\dots a_{k}\sigma\dots\sigma}\partial_{a_{1}}\dots\partial_{a_{k}}\partial_{\bar{b}}\partial_{\bar{b}_{1}}\dots\partial_{\bar{b}_{l}}
\partial_{\sigma}\dots\partial_{\sigma}
\end{gather*}
coming from commuting $V^{a_{1}\dots a_{k}\bar{b}_{1}\dots\bar{b}_{l}\sigma\dots\sigma}$ with derivatives of the sub-Laplacian. Commuting the derivatives of the sub-Laplacian with the derivatives of $P$, we get
\begin{gather*}
i(k-l)V^{a_{1}\dots a_{k}\bar{b}_{1}\dots\bar{b}_{l}\sigma\dots\sigma}\partial_{a_{1}}\dots\partial_{a_{k}}\partial_{\bar{b}_{1}}\dots\partial_{\bar{b}_{l}}
\partial_{\sigma}\dots\partial_{\sigma}\partial_{\sigma}
\end{gather*}
So together the leading term of the commutator is
\begin{gather*}
\sum_{k+l=d+1}(\partial^{(a_{1}}V^{a_{2}\dots a_{k})\bar{b}_{1}\dots\bar{b}_{l}}+\partial^{(\bar{b}_{1}}V^{\bar{b}_{2}\dots\bar{b}_{l})a_{1}\dots a_{k}})\partial_{a_{1}}\dots\partial_{a_{k}}\partial_{\bar{b}_{1}}\dots\partial_{\bar{b}_{l}}+\\
+\sum_{1\leq k\leq d}(ikV^{a_{1}\dots a_{k}\sigma\dots\sigma}+\partial^{(a_{1}}V^{a_{2}\dots a_{k})\sigma\dots\sigma})\partial_{a_{1}}\dots\partial_{a_{k}}\partial_{\sigma}\dots\partial_{\sigma}+\\
+\sum_{1\leq l\leq d}(-ilV^{\bar{b}_{1}\dots\bar{b}_{k}\sigma\dots\sigma}+\partial^{(\bar{b}_{1}}V^{\bar{b}_{2}\dots\bar{b}_{l})\sigma\dots\sigma})\partial_{\bar{b}_{1}}\dots\partial_{\bar{b}_{l}}\partial_{\sigma}\dots\partial_{\sigma}+\\
+\sum_{\latop{k\geq1,l\geq1}{k+l\leq d}}i(k-l)V^{a_{1}\dots a_{k}\bar{b}_{1}\dots\bar{b}_{l}\sigma\dots\sigma}\partial_{a_{1}}\dots\partial_{a_{k}}\partial_{\bar{b}_{1}}\dots\partial_{\bar{b}_{l}}\partial_{\sigma}\dots\partial_{\sigma}+\\
+\sum_{\latop{k\geq1,l\geq1}{k+l\leq d}}(\partial^{(a_{1}}V^{a_{2}\dots a_{k})\bar{b}_{1}\dots\bar{b}_{l}\sigma\dots\sigma}+\partial^{(\bar{b}_{1}}V^{\bar{b}_{2}\dots\bar{b}_{l})a_{1}\dots a_{k}\sigma\dots\sigma})\partial_{a_{1}}\dots\partial_{a_{k}}\partial_{\bar{b}_{1}}\dots\partial_{\bar{b}_{l}}\partial_{\sigma}\dots\partial_{\sigma}
\end{gather*}
where the second and third row are special cases of the last two rows for $l=0$ and $k=0$, respectively. This should be the leading term of some operator of the form $\delta\Delta$. This is only possible, if
\begin{gather}\label{abara}
\partial^{(a_{1}}V^{a_{2}\dots a_{k})\bar{b}_{1}\dots\bar{b}_{l}}+\partial^{(\bar{b}_{1}}V^{\bar{b}_{2}\dots\bar{b}_{l})a_{1}\dots a_{k}}=g^{(a_{1}\bar{b}_{1}}\lambda^{a_{2}\dots a_{k}\bar{b}_{1}\dots\bar{b}_{l})}
\end{gather}
for some tensor $\lambda$ and $k,l\geq1$, $k+l=d+1$,
\begin{gather}\label{a,bara}
\partial^{(a_{1}}V^{a_{2}\dots a_{d+1})}=0\qquad\partial^{(\bar{b}_{1}}V^{\bar{b}_{2}\dots\bar{b}_{d+1})}=0
\end{gather}
\begin{gather}\label{asigma}
ikV^{a_{1}\dots a_{k}\sigma\dots\sigma}+\partial^{(a_{1}}V^{a_{2}\dots a_{k})\sigma\dots\sigma}=0
\end{gather}
for $1\leq k\leq d$,
\begin{gather}\label{barasigma}
-ilV^{\bar{b}_{1}\dots\bar{b}_{k}\sigma\dots\sigma}+\partial^{(\bar{b}_{1}}V^{\bar{b}_{2}\dots\bar{b}_{l})\sigma\dots\sigma}=0
\end{gather}
for $1\leq l\leq d$,
\begin{gather}\label{abarasigma}
i(k-l)V^{a_{1}\dots a_{k}\bar{b}_{1}\dots\bar{b}_{l}\sigma\dots\sigma}+\partial^{(a_{1}}V^{a_{2}\dots a_{k})\bar{b}_{1}\dots\bar{b}_{l}\sigma\dots\sigma}+\partial^{(\bar{b}_{1}}V^{\bar{b}_{2}\dots\bar{b}_{l})a_{1}\dots a_{k}\sigma\dots\sigma}=\\
=g^{(a_{1}|(\bar{b}_{1}}\lambda^{|a_{2}\dots a_{k})|\bar{b}_{1}\dots\bar{b}_{l})\sigma\dots\sigma}\nonumber
\end{gather}
for $1\leq k$, $1\leq l$, $k+l\leq d$ and some tensor $\lambda$.\par
We will proceed by induction. From equations (\ref{asigma}), (\ref{barasigma}) and (\ref{a,bara}) we see that
\begin{gather}\label{kerbgg1}
\partial^{(a_{1}}\dots\partial^{a_{d+1})}V^{\sigma\dots\sigma}=0\\
\partial^{(\bar{b}_{1}}\dots\partial^{\bar{b}_{d+1})}V^{\sigma\dots\sigma}=0\nonumber
\end{gather}
what is exactly that $V^{\sigma\dots\sigma}$ lies in the kernel of first BGG operator corresponding to $\oooo{d}{0}{0}{d}$. We also see that the terms $V^{a_{1}\dots a_{k}\sigma\dots\sigma}$ and $V^{\bar{b}_{1}\dots\bar{b}_{k}\sigma\dots\sigma}$ only depend on $V^{\sigma\dots\sigma}$. To compute the dependence explicitly, we use $k$ times the equation (\ref{asigma}) and (\ref{barasigma}), respectively. We get
\begin{gather*}
V^{a_{1}\dots a_{k}\sigma\dots\sigma}=\frac{i^{k}}{k!}\partial^{(a_{1}}\dots\partial^{a_{k})}V^{\sigma\dots\sigma}\\
V^{\bar{b}_{1}\dots\bar{b}_{k}\sigma\dots\sigma}=\frac{(-i)^{k}}{k!}\partial^{(\bar{b}_{1}}\dots\partial^{\bar{b}_{k})}V^{\sigma\dots\sigma}
\end{gather*}
Putting $V^{\sigma\dots\sigma}=0$ (this implies $V^{a_{1}\dots a_{k}\bar{b}_{1}\dots\bar{b}_{l}\sigma\dots\sigma}=0$ for $\mbox{min}(k,l)=0$), from equations (\ref{abarasigma}) we see that
\begin{gather}\label{kerbgg2}
\mbox{the trace-free part of}\quad\partial^{(a_{1}}\dots\partial^{a_{d-1}}V^{a_{d})\bar{b}_{1}\sigma\dots\sigma}=0\\
\mbox{the trace-free part of}\quad\partial^{(\bar{b}_{1}}\dots\partial^{\bar{b}_{d-1}}V^{\bar{b}_{d})a_{1}\sigma\dots\sigma}=0\nonumber
\end{gather}
what is exactly that $V^{a_{1}\bar{b}_{1}\sigma\dots\sigma}$ lies in the kernel of the first BGG operator corresponding to $\quad\oooo{d-2}{1}{1}{d-2}\quad$. We also see that the terms $V^{a_{1}\dots a_{k}\bar{b}_{1}\dots\bar{b}_{l}\sigma\dots\sigma}$ with $\mbox{min}(k,l)=1$ only depend on $V^{a_{1}\bar{b}_{1}\sigma\dots\sigma}$. To compute the dependence explicitly, we use $k$ times equation (\ref{abarasigma}). We get
\begin{gather*}
V^{a_{1}\dots a_{k+1}\bar{b}_{1}\sigma\dots\sigma}=\frac{i^{k}}{k!}\partial^{(a_{1}}\dots\partial^{a_{k}}V^{a_{k+1})\bar{b}_{1}\sigma\dots\sigma}\\
V^{a_{1}\bar{b}_{1}\dots\bar{b}_{k+1}\sigma\dots\sigma}=
\frac{(-i)^{k}}{k!}\partial^{(\bar{b}_{1}}\dots\partial^{\bar{b}_{k}}V^{\bar{b}_{k+1})a_{1}\sigma\dots\sigma}
\end{gather*}
Continuing this way, we see for each $s$ such that $0<2s\leq d$, that putting $V^{\sigma\dots\sigma}=\dots=V^{a_{1}\dots a_{s-1}\bar{b}_{1}\dots\bar{b}_{s-1}\sigma\dots\sigma}=0$, we have $V^{a_{1}\dots a_{k}\bar{b}_{1}\dots\bar{b}_{l}\sigma\dots\sigma}=0$ for $\mbox{min}(k,l)<s$. From equations (\ref{abarasigma}) we see that
\begin{gather}\label{kerbgg3}
\mbox{the trace-free part of}\quad\partial^{(a_{1}}\dots\partial^{a_{d+1-2s}}V^{a_{d+2-2s}\dots a_{d+1-s})\bar{b}_{1}\dots\bar{b}_{s}\sigma\dots\sigma}=0\\
\mbox{the trace-free part of}\quad\partial^{(\bar{b}_{1}}\dots\partial^{\bar{b}_{d+1-2s}}V^{\bar{b}_{d+2-2s}\dots\bar{b}_{d+1-s}a_{1}\dots a_{s}\sigma\dots\sigma}=0\nonumber
\end{gather}
what is exactly the first BGG operator corresponding to $\quad\oooo{d-2s}{s}{s}{d-2s}\quad$. We also see that the terms $V^{a_{1}\dots a_{k}\bar{b}_{1}\dots\bar{b}_{l}\sigma\dots\sigma}$ with $\mbox{min}(k,l)=s$ only depend on $V^{a_{1}\dots a_{s}\bar{b}_{1}\dots\bar{b}_{s}\sigma\dots\sigma}$. To compute the dependence explicitly, we use $k$ times equation (\ref{abarasigma}). We get
\begin{gather*}
V^{a_{1}\dots a_{k+s}\bar{b}_{1}\dots\bar{b}_{s}\sigma\dots\sigma}=\frac{i^{k}}{k!}\partial^{(a_{1}}\dots\partial^{a_{k}}V^{a_{k+1}\dots a_{k+s})\bar{b}_{1}\dots\bar{b}_{s}\sigma\dots\sigma}\\
V^{a_{1}\dots a_{s}\bar{b}_{1}\dots\bar{b}_{k+s}\sigma\dots\sigma}=
\frac{(-i)^{k}}{k!}\partial^{(\bar{b}_{1}}\dots\partial^{\bar{b}_{k}}V^{\bar{b}_{k+1}\dots\bar{b}_{k+s})a_{1}\dots a_{s}\sigma\dots\sigma}
\end{gather*} 
\end{itemize}
\end{proof}
\begin{remark}
From the proof it is easy to see why we have introduced the finer filtration. It comes from the structure of PDE's for the symbol.
\end{remark}
\begin{remark}
The representation $\oooo{a}{b}{b}{a}$ of $SU(p+1,g+1)$ is simply the Cartan product of $\oooo{a}{b}{0}{0}$ with its dual/conjugate. The representation $\oooo{a}{b}{0}{0}$ is an irreducible subrepresentation of $\otimes^{d}\mathbb{C}^{n+2}$ with symmetries given by Young tableau
\begin{displaymath}
\tab
\end{displaymath} 
with $b$ columns with two boxes and $a$ columns with one box (the total number of boxes is $d$).
\end{remark}

\section{Ambient construction}

\subsection{Ambient space}

\begin{definition}
Let $M$ be as above. The \emph{ambient space} for $M$ is $\mathbb{C}^{n+2}$ $(z^{0},z^{1},\dots,z^{n},z^{\infty})$ with non-degenerate Hermitean metric $g_{\bar{A}B}$ of the form
\begin{displaymath}
\left(\begin{matrix}
0&0&1\\
0&g_{\bar{a}b}&0\\
1&0&0
\end{matrix}\right)
\end{displaymath}
We will denote
\begin{displaymath}
x^{A}=\left(\begin{matrix}
x^{0}\\x^{a}\\x^{\infty} 
\end{matrix}\right)
\end{displaymath}
\end{definition}
The term \emph{ambient} will be used when referring to the objects defined on some open subset of $\mathbb{C}^{n+2}$. The ambient Laplace operator will be distinguished by tilde $\tilde{\Delta}f=g^{A\bar{B}}\partial_{A}\partial_{\bar{B}}$.
\begin{definition}
Let
\begin{equation}\label{r}
r=g_{A\bar{B}}x^{A}x^{\bar{B}}
\end{equation}
be the quadratic form associated to the ambient metric $g_{A\bar{B}}$. The \emph{null cone} $\mathcal{N}$ is the zero set of $r$.
\begin{displaymath}
\mathcal{N}=\left\{x\in\mathbb{C}^{n+2}|r(x)=0\right\}
\end{displaymath}
\end{definition}
Now consider the mapping $\phi:M\rightarrow\mathbb{C}^{n+2}$ given by
\begin{displaymath}
(z^{a},i\sigma)\mapsto\left(\begin{matrix}
1\\z^{a}\\-\frac{z^{a}z_{a}}{2}+i\sigma
\end{matrix}\right)=:\phi^{A}
\end{displaymath}
The mapping $\phi$ is actually a restriction to $M$ of the embedding $\imath:\mathbb{C}^{n+1}\rightarrow\mathbb{C}^{n+2}$ given by $(z^{1},\dots,z^{n},z^{\infty})\mapsto(1,z^{1},\dots,z^{n},z^{\infty})$. It is easily seen that $\phi(M)$ lies on the null cone and that this characterizes $M$ in $\mathbb{C}^{n+1}$.
\begin{definition}
Let $z^{0}\in\mathbb{C}$, $\rho,\sigma\in\mathbb{R}$ and $z^{a}\in\mathbb{C}^{n}$.\\
\parbox{4cm}{\begin{displaymath}
X^{A}=\left(\begin{matrix}
z^{0}\\z^{0}z^{a}\\z^{0}(\rho-\frac{z^{a}z_{a}}{2}+i\sigma)
\end{matrix}\right)
\end{displaymath}}
\hfill\parbox{4cm}{\begin{displaymath}
Y^{A}_{b}=\partial_{b}X^{A}=\left(\begin{matrix}
0_{b}\\z^{0}\delta^{a}_{b}\\-z^{0}z_{b}
\end{matrix}\right)
\end{displaymath}}\\
\parbox{4cm}{\begin{displaymath}
Y^{A}_{\bar{b}}=\partial_{\bar{b}}X^{A}=0
\end{displaymath}}
\hfill\parbox{5cm}{\begin{displaymath}
Z^{A}=-\frac{1}{n}\partial^{b}Y^{A}_{b}=\left(\begin{matrix}
0\\0^{a}\\z^{0}
\end{matrix}\right)
\end{displaymath}}\\
Similarly for $X^{\bar{A}}$, $Y^{\bar{A}}_{b}$, $Y^{\bar{A}}_{\bar{b}}$, $Z^{\bar{A}}$, $X_{A}$, $Y^{b}_{A}$, $Y^{\bar{b}}_{A}$, $Z_{A}$, $X_{\bar{A}}$, $Y^{b}_{\bar{A}}$, $Y^{\bar{b}}_{\bar{A}}$, $Z_{\bar{A}}$, and other variations, respectively.
\end{definition}
\begin{lemma}
\begin{equation}\label{id}
|z^{0}|^{2}\delta^{A}_{B}=(X^{A}-\rho Z^{A})Z_{B}+Z^{A}(X_{B}-\rho Z_{B})+Y^{A}_{c}Y^{c}_{B}
\end{equation}
\end{lemma}
The mapping
\begin{displaymath}
\Phi(z^{0},z^{a},z^{\infty})=\left(\begin{matrix}
z^{0}\\z^{0}z^{a}\\z^{0}(z^{\infty}-\frac{z^{a}z_{a}}{2})
\end{matrix}\right)=\left(\begin{matrix}
y^{0}\\y^{a}\\y^{\infty}
\end{matrix}\right)
\end{displaymath}
where $z^{\infty}=\rho+i\sigma$, defines a change of coordinates, which, however, is smooth, but not holomorphic. We see that $\phi(z^{a},i\sigma)=\Phi(1,z^{a},i\sigma)$ and the identity (\ref{id}) simplifies on the image of $\phi$ to
\begin{displaymath}
\delta^{A}_{B}=X^{A}Z_{B}+Z^{A}X_{B}+Y^{A}_{c}Y^{c}_{B}
\end{displaymath}
Similarly for $\delta^{\bar{A}}_{\bar{B}}$:
\begin{displaymath}
\delta^{\bar{A}}_{\bar{B}}=X^{\bar{A}}Z_{\bar{B}}+Z^{\bar{A}}X_{\bar{B}}+Y^{\bar{A}}_{\bar{c}}Y^{\bar{c}}_{\bar{B}}
\end{displaymath}
\begin{lemma}
The operator $\mathbb{E}=x^{C}\partial_{C}$ in the new coordinates is equal to
\begin{equation}
\mathbb{E}=z^{0}\partial_{z^{0}}
\end{equation}
\end{lemma}
\begin{proof}
For $f(y^{A},y^{\bar{A}})\in\mathcal{C}^{\infty}(\mathbb{C}^{n+2})$ we have
\begin{eqnarray*}
z^{0}\partial_{z^{0}}f&=&z^{0}(\frac{\partial f}{\partial y^{0}}+z^{a}\frac{\partial f}{\partial y^{a}}+(z^{\infty}-\frac{z^{a}z_{a}}{2})
\frac{\partial f}{\partial y^{\infty}})\\
&=&y^{0}(\frac{\partial f}{\partial y^{0}}+\frac{y^{a}}{y^{0}}\frac{\partial f}{\partial y^{a}}+\frac{y^{\infty}}{y^{0}}
\frac{\partial f}{\partial y^{\infty}})\\
&=&(y^{A}\partial_{A}f)\circ\Phi
\end{eqnarray*}
\end{proof}
Here are some identities we will need later:
\begin{eqnarray}\label{der}
Y^{A}_{q}\partial_{A}Y^{\bar{B}}_{\bar{r}}=-g_{q\bar{r}}Z^{\bar{B}}&\qquad&
Y^{\bar{B}}_{\bar{r}}\partial_{\bar{B}}Y^{A}_{q}=-g_{q\bar{r}}Z^{A}\\
Z^{B}\partial_{B}=\partial_{z^{\infty}}&\qquad&Y^{A}_{q}\partial_{A}=\partial_{q}\nonumber\\
Z^{\bar{B}}\partial_{\bar{B}}=\partial_{z^{\bar{\infty}}}&\qquad&Y^{\bar{B}}_{\bar{r}}\partial_{\bar{B}}=\partial_{\bar{r}}\nonumber
\end{eqnarray}

\subsection{Ambient construction of the sub-Laplacian}

From now on (since in the second part we will use the representation theory of $SL(n+2,\mathbb{C})$) we will use the ambient metric to identify $\bar{V}$ with $V^{*}$, so we will replace all ambient upper (lower) barred indices by ambient lower (upper) indices. This will mean that $\mathbb{E}=X^{A}\partial_{A}$, $\bar{\mathbb{E}}=X_{A}\partial^{A}$, $\tilde{\Delta}=\partial_{A}\partial^{A}=\partial^{A}\partial_{A}$, and $r=X^{A}X_{A}=X_{A}X^{A}$, respectively.
\begin{definition}
Suppose that $F$ is a smooth complex-valued function defined on a neighbourhood of the origin in $M$. Then for any pair $(w_{1},w_{2})\in\mathbb{C}^{2}$, s.t. $w_{1}-w_{2}\in\mathbb{Z}$
\begin{equation}\label{F}
f(\Phi(z^{0},z^{a},i\sigma))=(z^{0})^{w_{1}}(z^{\bar{0}})^{w_{2}}F(z^{a},i\sigma)
\end{equation}
defines a smooth function on a conical neighbourhood of $(1,0,0)$ in the null-cone $\mathcal{N}$. Conversely, $F$ may be recovered from $f$ by setting $z^{0}=1$. Hence, for fixed $(w_{1},w_{2})$, the functions $F$ and $f$ are equivalent.
\end{definition}
\begin{remark}
If we view $\mathcal{N}\setminus\{x=(z^{0},\dots,z^{\infty})\in\mathcal{N}:z^{0}=0\}$ as a principal $\mathbb{C}^{\times}$-bundle over $M$, then we can represent the sections of $\mathcal{E}(w_{1},w_{2})$ as equivariant functions on the null-cone. But these are exactly the $(w_{1},w_{2})$-homogeneous functions as defined above. Having in mind that we are working with the very flat Weyl structure on $M$, we can identify densities of arbitrary weights with functions when working on $M$.
\end{remark}
We want to use the ambient construction to represent differential operators on $M$ by much simpler ambient differential operators. In order to be able to apply ambient differential operators to $f$, we need to extend it from the null-cone to the whole space or at least to some open (in $\mathbb{C}^{n+2}$) neighbourhood of $(1,0,0)$. There are infinitely many choices for such an extension even if we restrict ourselves to the homogeneous ones. Nevertheless, any two such extensions will differ by a very convenient factor.
\begin{lemma}
Let $f$ and $\hat{f}$ be two smooth $(w_{1},w_{2})$-homogeneous extensions of $F$ on some open neighbourhood of $(1,0,0)$. Then there exists a smooth $(w_{1}-1,w_{2}-1)$-homogeneous function $h$such that $(f-\hat{f})(y^{A})=r(y^{A})h(y^{A})$, where $r$ is defined by (\ref{r}).
\end{lemma}
\begin{proof}
If we perform coordinate transformation
\begin{displaymath}
(y^{0},y^{a},y^{\infty})\mapsto(y^{0},y^{a},r+ip)=(y^{0},y^{a},2y^{0}y^{\bar{\infty}}+y^{a}y_{a})
\end{displaymath}
we will be dealing with 2 functions equal on the  real hyperplane $r=0$. For any smooth complex-valued function $k$ on $\mathbb{C}^{n+2}$ holds
\begin{eqnarray*}
k(y^{0},y^{a},r+ip)=k(y^{0},y^{a},ip)+\int^{1}_{0}\frac{d}{dt}k(y^{0},y^{a},tr+ip)dt=\\
=k(y^{0},y^{a},ip)+r\int^{1}_{0}\frac{\partial k}{\partial(r+ip)}(y^{0},y^{a},tr+ip)+\frac{\partial k}{\partial(r-ip)}(y^{0},y^{a},tr+ip)dt
\end{eqnarray*}
So if we take $k$ as the difference of two $(w_{1},w_{2})$-homogeneous extensions of $F$, we will have $k(y^{0},y^{a},ip)=0$ and thus it follows that $f-\hat{f}=rh$. This $h$ has homogeneity $(w_{1}-1,w_{2}-1)$, because $r$ has homogeneity $(1,1)$. 
\end{proof}
\begin{remark}\label{chain}
The classical chain rule formula gives
\begin{eqnarray*}
\partial_{a}F=\partial_{a}(f\circ\phi)=(\partial_{a}\phi^{B}\partial_{B}f+\partial_{a}\phi_{B}\partial^{B}f)\circ\phi=\\
(Y^{B}_{a}\partial_{B}f)\circ\phi=(\partial_{a}f)\circ\phi\\
\partial_{\bar{a}}F=\partial_{\bar{a}}(f\circ\phi)=(\partial_{\bar{a}}\phi^{B}\partial_{B}f+\partial_{\bar{a}}\phi_{B}\partial^{B}f)\circ\phi=\\
(Y_{B\bar{a}}\partial^{B}f)\circ\phi=(\partial_{\bar{a}}f)\circ\phi\\
\partial_{\sigma}F=\partial_{\sigma}(f\circ\phi)=(\partial_{\sigma}\phi^{A}\partial_{A}f+\partial_{\sigma}\phi_{A}\partial^{A}f)\circ\phi=\\
(i\partial_{z^{\infty}}f-i\partial_{z^{\bar{\infty}}}f)\circ\phi=(\partial_{\sigma}f)\circ\phi
\end{eqnarray*}
\end{remark}
\begin{lemma}
For homogeneous function $h$ on $\mathbb{C}^{n+2}$ of bidegree $(w_{1}-1,w_{2}-1)$ holds
\begin{displaymath}
\tilde{\Delta}(rh)=r\tilde{\Delta}h+(n+w_{1}+w_{2})h
\end{displaymath}
\end{lemma}
\begin{proof}
\begin{eqnarray*}
\tilde{\Delta}(rh)&=&\partial_{A}\partial^{A}(rh)=\partial_{A}(x^{A}h+r\partial^{A}h)=\\
&=&(n+2)h+x^{A}\partial_{A}h+\delta^{B}_{A}x_{B}\partial^{B}h+r\tilde{\Delta}h=\\
&=&(n+2)h+(w_{1}-1)h+(w_{2}-1)h+r\tilde{\Delta}h=\\
&=&r\tilde{\Delta}h+(n+w_{1}+w_{2})h
\end{eqnarray*}
\end{proof}
It immediately follows that for $n+w_{1}+w_{2}=0$, then $\tilde{\Delta}f|_{\mathcal{N}}$ depends only on the restriction of $f$ to the null-cone and hence it depends only on $F$. This defines a differential operator on $M$.
\begin{theorem}
Let $F$ be a smooth complex-valued function on some open neighbourhood of $0\in M$ and let $f$ be the smooth homogeneous function of bidegree $(w_{1},w_{2})$ that corresponds to $F$ via (\ref{F}) and is defined on some open neighbourhood of $(1,0,0)\in\mathbb{C}^{n+2}$. Then the following equality holds
\begin{displaymath}
(\tilde{\Delta}f)\circ\phi=\Delta F
\end{displaymath} 
where $\Delta$ is the CR sub-Laplacian.
\end{theorem}
\begin{proof}
Using the equation (\ref{id}) we obtain
\begin{gather*}
(\partial_{A}\partial^{A}f)\circ\phi=\\
=[(X^{A}Z_{B}+Z^{A}X_{B})\partial_{A}\partial^{B}f+Y^{A}_{q}Y_{B}^{q})\partial_{A}\partial^{B}f]\circ\phi
\end{gather*}
Now the first term gives
\begin{gather*}
(X^{A}Z_{B}+Z^{A}X_{B})\partial_{A}\partial^{B}=\\
-Z^{A}(\partial_{A}X_{B})\partial^{B}+Z^{A}\partial_{A}X_{B}\partial^{B}-Z_{B}(\partial^{B}X^{A})\partial_{A}+Z_{B}\partial^{B}X^{A}\partial_{A}=\\
Z^{A}\partial_{A}X_{B}\partial^{B}+Z_{B}\partial^{B}X^{A}\partial_{A}=\\
Z^{A}\partial_{A}\bar{\mathbb{E}}+Z_{B}\partial^{B}\mathbb{E}
\end{gather*}
applied to $f$ and evaluated on the image of $\phi$. Let us recall that $Z^{A}\partial_{A}=\partial_{z^{\infty}}$ and $Z_{B}\partial^{B}=\partial_{z^{\bar{\infty}}}$. The second term is
\begin{gather*}
\left(g^{q\bar{r}}Y^{A}_{q}Y_{B\bar{r}}\partial_{A}\partial^{B}f\right)\circ\phi=\\
=\left(g^{q\bar{r}}[Y^{A}_{q}\partial_{A}Y_{B\bar{r}}\partial^{B}-Y^{A}_{q}(\partial_{A}Y_{B\bar{r}})\partial^{B}]f\right)\circ\phi=\\
=\left(g^{q\bar{r}}Y^{A}_{q}\partial_{A}Y_{B\bar{r}}\partial^{B}f+\frac{n}{2}Z_{B}\partial^{B}f\right)\circ\phi=\\
=\left(g^{q\bar{r}}\partial_{q}\partial_{\bar{r}}f+\frac{n}{2}\partial_{z^{\bar{\infty}}}f\right)\circ\phi
\end{gather*}
Another way to compute the second term is
\begin{gather*}
\left(g^{q\bar{r}}Y_{B\bar{r}}Y^{A}_{q}\partial^{B}\partial_{A}f\right)\circ\phi=\\
=\left(g^{q\bar{r}}[Y_{B\bar{r}}\partial^{B}Y^{A}_{q}\partial_{A}-Y_{B\bar{r}}(\partial^{B}Y^{A}_{q})\partial_{A}]f\right)\circ\phi=\\
=\left(g^{q\bar{r}}Y_{B\bar{r}}\partial^{B}Y^{A}_{q}\partial_{A}f+\frac{n}{2}Z^{A}\partial_{A}f\right)\circ\phi=\\
=\left(g^{q\bar{r}}\partial_{\bar{r}}\partial_{q}f+\frac{n}{2}\partial_{z^{\infty}}f\right)\circ\phi
\end{gather*}
We will take as the second term one half of their sum
\begin{displaymath}
\left([\frac{1}{2}g^{q\bar{r}}(\partial_{\bar{r}}\partial_{q}+\partial_{q}\partial_{\bar{r}})+\frac{n}{2}(\partial_{z^{\bar{\infty}}}+\partial_{z^{\infty}})]f\right)\circ\phi
\end{displaymath}
Altogether we get
\begin{gather*}
(\tilde{\Delta}f)\circ\phi=\left(w_{1}\partial_{z^{\bar{\infty}}}f+w_{2}\partial_{z^{\infty}}f\right)\circ\phi+\\
+\left([\frac{g^{q\bar{r}}}{2}(\partial_{\bar{r}}\partial_{q}+\partial_{q}\partial_{\bar{r}})+\frac{n}{2}(\partial_{z^{\bar{\infty}}}+\partial_{z^{\infty}})]f\right)\circ\phi=\\
=\left([\frac{g^{q\bar{r}}}{2}(\partial_{q}\partial_{\bar{r}}+\partial_{\bar{r}}\partial_{q})+\frac{n+w_{1}+w_{2}}{2}\partial_{\rho}+i\frac{w_{1}-w_{2}}{2}\partial_{\sigma}]f\right)\circ\phi=\\
=\frac{1}{2}(\partial^{a}\partial_{a}+\partial_{a}\partial^{a})F+\frac{i(w_{1}-w_{2})}{2}\partial_{\sigma}F
\end{gather*}
which completes the proof. We have only used the fact that $z^{\infty}=\rho+i\sigma$.
\end{proof}

\subsection{Ambient construction of symmetries - existence}

In previous subsection we have used simpler ambient operator to induce the sub-Laplacian on $M$. We want to use the same procedure for symmetries of the sub-Laplacian.\par
Before constructing the symmetries, we first replace 
\begin{displaymath}
V^{a_{1}\dots a_{k}\bar{b}_{1}\dots\bar{b}_{l}\sigma\dots\sigma}\partial_{a_{1}}\dots\partial_{a_{k}}\partial_{\bar{b}_{1}}\dots\partial_{\bar{b}_{l}}\partial_{\sigma}\dots\partial_{\sigma}
\end{displaymath} 
by 
\begin{displaymath}
V^{a_{1}\dots a_{k}\sigma\dots\sigma}_{b_{1}\dots b_{l}}\partial_{a_{1}}\dots\partial_{a_{k}}\partial^{b_{1}}\dots\partial^{b_{l}}\partial_{\sigma}\dots\partial_{\sigma}
\end{displaymath}
For these functions we rewrite the equations (\ref{abara}), (\ref{a,bara}), (\ref{asigma}), (\ref{barasigma}) and (\ref{abarasigma}) as
\begin{gather}\label{abara*}
\partial^{(a_{1}}V^{a_{2}\dots a_{k})}_{b_{1}\dots b_{l}}+\partial_{(b_{1}}V^{a_{1}\dots a_{k}}_{b_{2}\dots b_{l})}=\delta^{(a_{1}}_{(b_{1}}\lambda^{a_{2}\dots a_{k})}_{b_{2}\dots b_{l})}
\end{gather}
for some tensor $\lambda$ and $k,l\geq1$, $k+l=d+1$,
\begin{gather}\label{a,bara*}
\partial^{(a_{1}}V^{a_{2}\dots a_{d+1})}=0\qquad\partial_{(b_{1}}V_{b_{2}\dots b_{d+1})}=0
\end{gather}
\begin{gather}\label{asigma*}
ikV^{a_{1}\dots a_{k}\sigma\dots\sigma}+\partial^{(a_{1}}V^{a_{2}\dots a_{k})\sigma\dots\sigma}=0
\end{gather}
for $1\leq k\leq d$,
\begin{gather}\label{barasigma*}
-ilV^{\sigma\dots\sigma}_{b_{1}\dots b_{l}}+\partial_{(b_{1}}V^{\sigma\dots\sigma}_{b_{2}\dots b_{l})}=0
\end{gather}
for $1\leq l\leq d$,
\begin{gather}\label{abarasigma*}
i(k-l)V^{a_{1}\dots a_{k}\sigma\dots\sigma}_{b_{1}\dots b_{l}}+\partial^{(a_{1}}V^{a_{2}\dots a_{k})\sigma\dots\sigma}_{b_{1}\dots b_{l}}+\partial_{(b_{1}}V^{a_{1}\dots a_{k}\sigma\dots\sigma}_{b_{2}\dots b_{l})}=\nonumber\\
=\delta^{(a_{1}}_{(b_{1}}\lambda^{a_{2}\dots a_{k})\sigma\dots\sigma}_{b_{2}\dots b_{l}}
\end{gather}
for $1\leq k$, $1\leq l$, $k+l\leq d$ and some tensor $\lambda$. The corresponding first BGG equations are
\begin{gather}\label{kerbgg*}
\mbox{the trace-free part of}\quad\partial^{(a_{1}}\dots\partial^{a_{d+1-2s}}V^{a_{d+2-2s}\dots a_{d+1-s})\sigma\dots\sigma}_{b_{1}\dots b_{s}}=0\\
\mbox{the trace-free part of}\quad\partial_{(b_{1}}\dots\partial_{b_{d+1-2s}}V^{a_{1}\dots a_{s}\sigma\dots\sigma}_{b_{d+2-2s}\dots b_{d+1-s}}=0\nonumber
\end{gather}
\begin{lemma}
The first order operators $x^{A}\partial_{B}-x_{B}\partial^{A}$ commute with $\tilde{\Delta}$ and with $r$.
\end{lemma}
\begin{proof}
\begin{gather*}
\partial_{C}\partial^{C}(x^{A}\partial_{B}-x_{B}\partial^{A})=\\
=\partial_{C}(x^{A}\partial_{B}\partial^{C}-\delta^{C}_{B}\partial^{A}-x_{B}\partial^{A}\partial^{C})=\\
=\delta^{A}_{C}\partial_{B}\partial^{C}+x^{A}\partial_{B}\partial_{C}\partial^{C}-\partial_{B}\partial^{A}-x_{B}\partial^{A}\partial_{C}\partial^{C}=\\
=(x^{A}\partial_{B}-x_{B}\partial^{A})\partial_{C}\partial^{C}
\end{gather*}
Similarly, using that $r=x^{A}x_{A}$,
\begin{gather*}
(x^{A}\partial_{B}-x_{B}\partial^{A})x^{C}x_{C}=\\
x^{A}x_{B}+x^{A}x^{C}x_{C}\partial_{B}-x_{B}x^{A}-x_{B}x^{C}x_{C}\partial^{A}
\end{gather*}
\end{proof}
Now we know that any  complex linear combination of such operators, i.e. any operator of the form $V^{B}_{A}(x^{A}\partial_{B}-x_{B}\partial^{A})$ commutes with $\tilde{\Delta}$ and $r$, and hence induces a symmetry of the sub-Laplacian on $M$. The vector space of first order operators we have found so far, is clearly isomorphic to $\mathfrak{gl}(n+2,\mathbb{C})$ (the matrices $V^{B}_{A}$ are scalar). The operator corresponding to the central element is 
\begin{displaymath}
i(x^{B}\partial_{B}-x_{B}\partial^{B})
\end{displaymath}
which induces on $M$ scalar multiplication by $i(w_{1}-w_{2})$ on functions with weight $(w_{1},w_{2})$. Since scalar multiplication is not very interesting operator, we will restrict ourselves to operators corresponding to $\mathfrak{sl}(n+2,\mathbb{C})$.\par
Composing such first order operators, we get higher order operators with the same properties. Concretely, we may write them like this:
\begin{equation}\label{expr}
V^{B_{1}\dots B_{d}}_{A_{1}\dots A_{d}}(x^{A_{1}}\partial_{B_{1}}-x_{B_{1}}\partial^{A_{1}})\dots(x^{A_{d}}\partial_{B_{d}}-x_{B_{d}}\partial^{A_{d}})
\end{equation}
The expression (\ref{expr}) will be simplified and from the simplification we get symmetries of the tensor $V^{B_{1}\dots B_{d}}_{A_{1}\dots A_{d}}$.\par
First, we compute the commutator of two first order operators:
\begin{eqnarray}
&V^{B_{1}}_{A_{1}}W^{B_{2}}_{A_{2}}(x^{A_{1}}\partial_{B_{1}}-x_{B_{1}}\partial^{A_{1}})(x^{A_{2}}\partial_{B_{2}}-x_{B_{2}}\partial^{A_{2}})-&\nonumber\\
&-W^{B_{2}}_{A_{2}}V^{B_{1}}_{A_{1}}(x^{A_{2}}\partial_{B_{2}}-x_{B_{2}}\partial^{A_{2}})(x^{A_{1}}\partial_{B_{1}}-x_{B_{1}}\partial^{A_{1}})=&\nonumber\\
&=V^{B_{1}}_{A_{1}}W^{B_{2}}_{A_{2}}(x^{A_{1}}\delta^{A_{2}}_{B_{1}}\partial_{B_{2}}+x^{A_{1}}x^{A_{2}}\partial_{B_{1}}\partial_{B_{2}}-x^{A_{1}}x_{B_{2}}\partial_{B_{1}}\partial^{A_{2}})-&\nonumber\\
&-V^{B_{1}}_{A_{1}}W^{B_{2}}_{A_{2}}(x_{B_{1}}x^{A_{2}}\partial^{A_{1}}\partial_{B_{2}}-x_{B_{1}}\delta^{A_{1}}_{B_{2}}\partial^{A_{2}}-x_{B_{1}}x_{B_{2}}\partial^{A_{1}}\partial^{A_{2}})-&\nonumber\\
&-W^{B_{2}}_{A_{2}}V^{B_{1}}_{A_{1}}(x^{A_{2}}\delta^{A_{1}}_{B_{2}}\partial_{B_{1}}+x^{A_{2}}x^{A_{1}}\partial_{B_{2}}\partial_{B_{1}}-x^{A_{2}}x_{B_{1}}\partial_{B_{2}}\partial^{A_{1}})+&\nonumber\\
&+W^{B_{2}}_{A_{2}}V^{B_{1}}_{A_{1}}(x_{B_{2}}x^{A_{1}}\partial^{A_{2}}\partial_{B_{1}}-x_{B_{2}}\delta^{A_{2}}_{B_{1}}\partial^{A_{1}}-x_{B_{2}}x_{B_{1}}\partial^{A_{2}}\partial^{A_{1}})=&\nonumber\\
&=V^{C}_{A_{1}}W^{B_{2}}_{C}x^{A_{1}}\partial_{B_{2}}+V^{B_{1}}_{C}W^{C}_{A_{2}}x_{B_{1}}\partial^{A_{2}}
-W^{C}_{A_{2}}V^{B_{1}}_{C}x^{A_{2}}\partial_{B_{1}}-W^{B_{2}}_{C}V^{C}_{A_{1}}x_{B_{2}}\partial^{A_{1}}=&\nonumber\\
&=(V^{C}_{A}W^{B}_{C}-V^{B}_{C}W^{C}_{A})(x^{A}\partial_{B}-x_{B}\partial^{A})&
\end{eqnarray}
So we see that taking commutator does not enlarge the vector space of symmetries. Therefore we can restrict ourselves to such tensors $V^{B_{1}\dots B_{d}}_{A_{1}\dots A_{d}}$, which are symmetric in columns $B_{i}A_{i}$. We will want the induced operator to be of order $d$.\par
Let $I$, $J$ be ordered subsets of $\{1,\dots,d\}$ such that $I\cup J=\{1,\dots,d\}$ and $I\cap J=\emptyset$. Composing $d$ first order symmetries, we get
\begin{gather}\label{hos}
V^{B_{1}\dots B_{d}}_{A_{1}\dots A_{d}}\prod_{i=1}^{d}(X^{A_{i}}\partial_{B_{i}}-X_{B_{i}}\partial^{A_{i}})=\\
=\sum_{\latop{|I|=k,|J|=l}{k+l=d}}(-1)^{l}V^{B_{1}\dots B_{d}}_{A_{1}\dots A_{d}}X^{A_{i_{1}}}X^{A_{i_{k}}}X_{B_{j_{1}}}\dots X_{B_{j_{l}}}\partial_{B_{i_{1}}}\dots\partial_{B_{i_{k}}}\partial^{A_{j_{1}}}\dots\partial^{A_{j_{l}}}+LOTS\nonumber
\end{gather}
where $V^{B_{1}\dots B_{d}}_{A_{1}\dots A_{d}}$ is symmetric in columns $B_{i}A_{i}$. Since we want the induced operator to be of order $d$ (as we shall see, these operators will suffice), we may consider the tensor $V^{B_{1}\dots B_{d}}_{A_{1}\dots A_{d}}$ be totally trace-free (it is already trace-free within any column). Looking at the induced operator, we get
\begin{gather*}
\sum_{\latop{|I|=k,|J|=l}{k+l=d}}(-1)^{l}V^{B_{1}\dots B_{d}}_{A_{1}\dots A_{d}}X^{A_{i_{1}}}\dots X^{A_{i_{k}}}X_{B_{j_{1}}}\dots X_{B_{j_{l}}}\partial_{B_{i_{1}}}\dots\partial_{B_{i_{k}}}\partial^{A_{j_{1}}}\dots\partial^{A_{j_{l}}}=\\
=\sum_{\latop{|I|=k,|J|=l}{k+l=d}}(-1)^{l}V^{B_{1}\dots B_{d}}_{A_{1}\dots A_{d}}X^{A_{i_{1}}}\dots X^{A_{i_{k}}}X_{B_{j_{1}}}\dots X_{B_{j_{l}}}\delta^{D_{i_{1}}}_{B_{i_{1}}}\partial_{D_{i_{1}}}\dots\delta^{D_{i_{k}}}_{B_{i_{k}}}\partial_{D_{i_{k}}}\cdot\\
\cdot\delta^{A_{j_{1}}}_{C_{j_{1}}}\partial^{C_{j_{1}}}\dots\delta^{A_{j_{l}}}_{C_{j_{l}}}\partial^{C_{j_{l}}}=\\
\sum_{\latop{|I|=k,|J|=l}{k+l=d}}(-1)^{l}V^{B_{1}\dots B_{d}}_{A_{1}\dots A_{d}}X^{A_{i_{1}}}X^{A_{i_{k}}}X_{B_{j_{1}}}\dots X_{B_{j_{l}}}\cdot\\
\cdot(X^{D_{i_{1}}}Z_{B_{i_{1}}}+Z^{D_{i_{1}}}X_{B_{i_{1}}}+Y^{D_{i_{1}}}_{a_{1}}Y^{a_{1}}_{B_{i_{1}}})\partial_{D_{i_{1}}}\dots(X^{D_{i_{k}}}Z_{B_{i_{k}}}+Z^{D_{i_{k}}}X_{B_{i_{k}}}+Y^{D_{i_{k}}}_{a_{k}}Y^{a_{k}}_{B_{i_{k}}})\partial_{D_{i_{k}}}\cdot\\
\cdot(X^{A_{j_{1}}}Z_{C_{j_{1}}}+Z^{A_{j_{1}}}X_{C_{j_{1}}}+Y_{C_{j_{1}}}^{b_{1}}Y^{A_{j_{1}}}_{b_{1}})\partial^{C_{j_{1}}}\dots(X^{A_{j_{l}}}Z_{C_{j_{l}}}+Z^{A_{j_{l}}}X_{C_{j_{l}}}+Y_{C_{j_{l}}}^{b_{l}}Y^{A_{j_{l}}}_{b_{l}})\partial^{C_{j_{l}}}
\end{gather*}
Knowing that $Y^{D}_{a}\partial_{D}=\partial_{a}$, $Y^{b}_{C}\partial^{C}=\partial^{b}$, $X^{D}\partial_{D}=\mathbb{E}$, $X_{C}\partial^{C}=\bar{\mathbb{E}}$ and $Z^{D}\partial_{D}-Z_{C}\partial^{C}=-i\partial_{\sigma}$, we see we have to put
\begin{gather}\label{symbd}
V^{a_{1}\dots a_{d}}:=\sum_{\sigma\in\mathfrak{S}_{d}}\frac{1}{d!}V^{B_{1}\dots B_{d}}_{A_{1}\dots A_{d}}X^{A_{1}}\dots X^{A_{d}}Y^{a_{\sigma(1)}}_{B_{1}}\dots Y^{a_{\sigma(d)}}_{B_{d}}\circ\phi\\
V_{b_{1}\dots b_{d}}:=\sum_{\tau\in\mathfrak{S}_{d}}\frac{1}{d!}V^{B_{1}\dots B_{d}}_{A_{1}\dots A_{d}}X_{B_{1}}\dots X_{B_{d}}Y^{A_{1}}_{b_{\tau(1)}}\dots Y^{A_{d}}_{b_{\tau(d)}}\circ\phi\nonumber\\
V^{a_{1}\dots a_{k}\sigma\dots\sigma}:=\nonumber\\
=(-i)^{d-k}\sum_{|I|=k}\sum_{\sigma\in\mathfrak{S}_{k}}\frac{1}{k!}V^{B_{1}\dots B_{d}}_{A_{1}\dots A_{d}}X^{A_{1}}\dots X^{A_{d}}Y^{a_{\sigma(1)}}_{B_{i_{1}}}\dots Y^{a_{\sigma(k)}}_{B_{i_{k}}}X_{B}\dots X_{B}\circ\phi\nonumber\\
V_{b_{1}\dots b_{l}\sigma\dots\sigma}:=\nonumber\\
=(-1)^{l}(-i)^{d-l}\sum_{|J|=l}\sum_{\tau\in\mathfrak{S}_{l}}\frac{1}{l!}V^{B_{1}\dots B_{d}}_{A_{1}\dots A_{d}}X_{B_{1}}\dots X_{B_{d}}Y^{A_{j_{1}}}_{b_{\tau(1)}}\dots Y^{A_{j_{l}}}_{b_{\tau(l)}}X^{A}\dots X^{A}\circ\phi\nonumber\\
V^{\sigma\dots\sigma}:=(-i)^{d}V^{B_{1}\dots B_{d}}_{A_{1}\dots A_{d}}X^{A_{1}}\dots X^{A_{d}}X_{B_{1}}\dots X_{B_{d}}\circ\phi\nonumber\\
V^{a_{1}\dots a_{k}\sigma\dots\sigma}_{b_{1}\dots b_{l}}:=\nonumber\\
=(-1)^{l}(-i)^{d-k-l}\sum_{\latop{|I|=k,|J|=l}{I\cap J=\emptyset}}\sum_{\latop{\sigma\in\mathfrak{S}_{k}}{\tau\in\mathfrak{S}_{l}}}\frac{1}{k!l!}V^{B_{1}\dots B_{d}}_{A_{1}\dots A_{d}}X^{A_{i_{1}}}\dots X^{A_{i_{k}}}X_{B_{j_{1}}}\dots X_{B_{j_{l}}}\cdot\nonumber\\
\cdot Y^{a_{\sigma(1)}}_{B_{i_{1}}}\dots Y^{a_{\sigma(k)}}_{B_{i_{k}}}Y^{A_{j_{1}}}_{b_{\tau(1)}}\dots Y^{A_{j_{l}}}_{b_{\tau(l)}}X^{A}\dots X^{A}X_{B}\dots X_{B}\circ\phi\nonumber
\end{gather}
where $I$ and $J$ are ordered subsets of $\{1,\dots,d\}$ and all indices not specified are simply the remaining indices. We will always assume this in the sequel. These functions satisfy the conditions (\ref{abara*}), (\ref{a,bara*}), (\ref{asigma*}), (\ref{barasigma*}) and (\ref{abarasigma*}).
\begin{proposition}\label{prop1}
For every given $d$ and $s$, $0\leq 2s\leq d$ and any (weighted) tensor $T^{a_{1}\dots a_{s}}_{b_{1}\dots b_{s}}$ satisfying the first BGG equation (\ref{kerbgg*}) there are canonically defined differential operators $P_{T}$ and $\delta_{T}$ of degree $(d,s)$ with the leading part (in the sense of degree) of the symbol being $V^{a_{1}\dots a_{s}\sigma\dots\sigma}_{b_{1}\dots b_{s}}=T^{a_{1}\dots a_{s}}_{b_{1}\dots b_{s}}$ such that $\Delta P_{T}=\delta_{T}\Delta$.
\end{proposition}
\begin{proof}
We know from above that every symmetry of order $d$ can be written as a sum of symmetries with $V^{a_{1}\dots a_{s}\sigma\dots\sigma}_{b_{1}\dots b_{s}}\neq0$ and $V^{a_{1}\dots a_{k}\sigma\dots\sigma}_{b_{1}\dots b_{k}}=0$ for all $k<s$ and $0\leq2s\leq d$. So it suffices to construct for each $0\leq2s\leq d$ and each function $V^{a_{1}\dots a_{s}\sigma\dots\sigma}_{b_{1}\dots b_{s}}$ satisfying the first BGG equation corresponding to \quad$\oooo{d-2s}{s}{s}{d-2s}$\quad a tensor $V^{B_{1}\dots B_{d}}_{A_{1}\dots A_{d}}$ inducing a symmetry with given (up to possible nonzero constant multiple) $V^{a_{1}\dots a_{s}\sigma\dots\sigma}_{b_{1}\dots b_{s}}$ and such that for all $k<s$ $V^{a_{1}\dots a_{k}\sigma\dots\sigma}_{b_{1}\dots b_{k}}=0$. Since contracting with $X$-s and $Y$-s is equivariant map and for each $s$ the set of possible $V^{a_{1}\dots a_{s}\sigma\dots\sigma}_{b_{1}\dots b_{s}}$-s (with $V^{a_{1}\dots a_{k}\sigma\dots\sigma}_{b_{1}\dots b_{k}}=0$ for all $k<s$) forms a complex irreducible representation of $SU(p+1,q+1)$ (and hence of $SL(n+1,\mathbb{C})$), it suffices to construct the tensor $V^{B_{1}\dots B_{d}}_{A_{1}\dots A_{d}}$ for one such function for each $s$.\par
For $s=0$ we put $V^{\sigma\dots\sigma}=1$ and the only nontrivial component of $V^{B_{1}\dots B_{d}}_{A_{1}\dots A_{d}}$ will be $V^{\infty\dots\infty}_{0\dots0}=1$. This does not depend on $d$ (only the constant factor does).\par
For $s>0$, we have the mapping
\begin{gather*}
V^{B_{1}\dots B_{d}}_{A_{1}\dots A_{d}}\mapsto V^{B_{1}\dots B_{d}}_{A_{1}\dots A_{d}}(-1)^{s}(-i)^{d-2s}\sum_{\latop{|I|=s,|J|=s}{I\cap J=\emptyset}}\sum_{\latop{\sigma\in\mathfrak{S}_{s}}{\tau\in\mathfrak{S}_{s}}}\frac{1}{(s!)^{2}}X^{A_{i_{1}}}\dots X^{A_{i_{s}}}\cdot\\
\cdot X_{B_{j_{1}}}\dots X_{B_{j_{s}}}X^{A}\dots X^{A}Y^{a_{\sigma(1)}}_{B_{i_{1}}}\dots Y^{a_{\sigma(s)}}_{B_{i_{s}}}Y_{b_{\tau(1)}}^{A_{j_{1}}}\dots Y_{b_{\tau(s)}}^{A_{j_{s}}}X_{B}\dots X_{B}\circ\phi=\\
=V^{a_{1}\dots a_{s}\sigma\dots\sigma}_{b_{1}\dots b_{s}}
\end{gather*}
We fix some \emph{constant} tensor field $V^{a_{1}\dots a_{s}\sigma\dots\sigma}_{b_{1}\dots b_{s}}$ and we put 
\begin{gather*}
V^{a_{1}\infty\dots a_{s}\infty\infty\dots\infty}_{0b_{1}\dots0b_{s}0\dots0}:=V^{a_{1}\dots a_{s}\sigma\dots\sigma}_{b_{1}\dots b_{s}}
\end{gather*}
This will be for now the only nonzero component up to symmetry in $B_{i}A_{i}$-s. This tensor surely satisfies the corresponding first BGG equation. It is easy to see that the symmetry induced by using this tensor has nonzero $V^{a_{1}\dots a_{s}\sigma\dots\sigma}_{b_{1}\dots b_{s}}$ (it is in fact, up to nonzero constant, our chosen one). But the symbol parts $V^{a_{1}\dots a_{k}\sigma\dots\sigma}_{b_{1}\dots b_{k}}$ for all $k<s$ are also nonzero. In order to make them vanish, we define some other components to be possibly nonzero. These components will not influence the symbol part $V^{a_{1}\dots a_{s}\sigma\dots\sigma}_{b_{1}\dots b_{s}}$.\par
Let's fix some $s>0$. We define
\begin{gather*}
V^{a_{1}\infty\dots a_{s}\infty\infty\dots\infty}_{0b_{1}\dots0b_{s}0\dots0}:=V^{a_{1}\dots a_{s}\sigma\dots\sigma}_{b_{1}\dots b_{s}}\\
V^{a_{1}\infty\dots a_{s-1}\infty a_{s}\infty\dots\infty}_{0b_{1}\dots0b_{s-1}b_{s}0\dots0}:=x_{1}V^{a_{1}\dots a_{s}\sigma\dots\sigma}_{b_{1}\dots b_{s}}\\
V^{a_{1}\infty\dots a_{s-2}\infty a_{s-1}a_{s}\infty\dots\infty}_{0b_{1}\dots0b_{s-2}b_{s-1}b_{s}0\dots0}:=x_{2}V^{a_{1}\dots a_{s}\sigma\dots\sigma}_{b_{1}\dots b_{s}}\\
\cdots\cdots\cdots\cdots\cdots\cdots\cdots\cdots\cdots\cdots\\
V^{a_{1}\dots a_{s}\infty\dots\infty}_{b_{1}\dots b_{s}0\dots0}:=x_{s}V^{a_{1}\dots a_{s}\sigma\dots\sigma}_{b_{1}\dots b_{s}}
\end{gather*}
These will be the only nonzero components of $V^{B_{1}\dots B_{d}}_{A_{1}\dots A_{d}}$ up to symmetry in $B_{i}A_{i}$-s. We will call them 'types'-in each row is one particular representative of one type. For brevity, we will write '$V^{a_{1}\dots a_{s}\sigma\dots\sigma}_{b_{1}\dots b_{s}}=1$'. To make the symbol parts $V^{a_{1}\dots a_{k}\sigma\dots\sigma}_{b_{1}\dots b_{k}}$ for all $k<s$ vanish, they have to satisfy the following system of linear equations:
\begin{gather}
\binom{d}{2s}\binom{2s}{s}a^{s}_{k+1,0}+\binom{d}{2s-2}\binom{2s-2}{s-1}\binom{d-2s+2}{1}a^{s}_{k+1,1}x_{1}+\\
+\binom{d}{2s-4}\binom{2s-4}{s-2}\binom{d-2s+4}{2}a^{s}_{k+1,2}x_{2}+\cdots+\binom{d}{0}\binom{0}{0}\binom{d}{s}a^{s}_{k+1,s}=0\nonumber
\end{gather}
Here the terms $\binom{d}{2s-2i}\binom{2s-2i}{s-i}\binom{d-2s+2i}{i}$ express the number of components of one type (it is always nonzero) and the numbers $a^{s}_{k+1,i}$ express the contribution of one component of corresponding type to the symbol part $V^{a_{1}\dots a_{k}\bar{b}_{1}\dots\bar{b}_{k}\sigma\dots\sigma}$ modulo the greatest common divisor, which is some polynomial. It is easy to see that 
\begin{gather*}
a^{s}_{k+1,i}=\binom{s-i}{k}\binom{s}{k}+\binom{s-i}{k-1}\binom{i}{1}\binom{s-1}{k}+\\
+\binom{s-i}{k-2}\binom{i}{2}\binom{s-2}{k}+\cdots+\binom{s-i}{k-j}\binom{i}{j}\binom{s-j}{k}+\cdots
\end{gather*}
This expression is always finite. So we have a system of linear equations in $x_{i}$-s, and we only need to prove the existence of some solution. We don't need the explicit expression. To prove the existence, we prove that the matrix of this system has nonzero determinant. First, the determinant is linear in columns, so it is a polynomial in $d$ times the determinant of matrix with entries $a^{s}_{k+1,i}$. We will prove that this last determinant is nonzero by induction on $s$.\par
We claim that $a^{s+1}_{k+2,i}-a^{s+1}_{k+2,i+1}=a^{s}_{k+1,i}$, where we put $a^{s}_{0,i}=0$. This means that
\begin{gather*}
\sum_{j\geq0}\binom{s+1-i}{k+1-j}\binom{i}{j}\binom{s+1-j}{k+1}=\\
=\sum_{j\geq0}\binom{s-i}{k-j}\binom{i}{j}\binom{s-j}{k}+\sum_{j=0}^{k+1}\binom{s-i}{k+1-j}\binom{i+1}{j}\binom{s+1-j}{k+1}
\end{gather*}
But
\begin{gather*}
\binom{s+1-i}{k+1-j}\binom{i}{j}\binom{s+1-j}{k+1}=\\
=\left(\binom{s-i}{k-j}+\binom{s-i}{k+1-j}\right)\binom{i}{j}\binom{s+1-j}{k+1}=\\
=\binom{s-i}{k+1-j}\binom{i}{j}\binom{s+1-j}{k+1}+\binom{s-i}{k-j}\binom{i}{j}\binom{s-j}{k}+\\
+\binom{s-i}{k-j}\binom{i}{j}\binom{s-j}{k+1}
\end{gather*}
Summing over $j$ we get
\begin{gather*}
\sum_{j\geq0}\binom{s+1-i}{k+1-j}\binom{i}{j}\binom{s+1-j}{k+1}=\\
=\sum_{j\geq0}\binom{s-i}{k-j}\binom{i}{j}\binom{s-j}{k}+\sum_{j=0}^{k+1}\binom{s-i}{k+1-j}\binom{i}{j}\binom{s+1-j}{k+1}+\\
+\sum_{j\geq1}\binom{s-i}{k-j}\binom{i}{j-1}\binom{s+1-j}{k+1}=\\
\sum_{j\geq0}\binom{s-i}{k-j}\binom{i}{j}\binom{s-j}{k}+\sum_{j=0}^{k+1}\binom{s-i}{k+1-j}\binom{i+1}{j}\binom{s+1-j}{k+1}
\end{gather*}
what is exactly what we claimed. Now we note that $a^{s}_{1,i}=1$ for all $s$ and $i$. Using this, we get
\begin{gather*}
\left|\begin{matrix}
1&1&\cdots&1&1\\
a^{s}_{2,1}&a^{s}_{2,2}&\cdots&a^{s}_{2,s-1}&a^{s}_{2,s}\\
\vdots&\vdots&\ddots&\vdots&\vdots\\
a^{s}_{s,1}&a^{s,2}&\cdots&a^{s}_{s,s-1}&a^{s}_{s,s}
\end{matrix}\right|=\left|\begin{matrix}
0&0&\cdots&0&1\\
a^{s}_{2,1}-a^{s}_{2,2}&a^{s}_{2,2}-a^{s}_{2,3}&\cdots&a^{s}_{2,s-1}-a^{s}_{2,s}&a^{s}_{2,s}\\
\vdots&\vdots&\ddots&\vdots&\vdots\\
a^{s}_{s,1}-a^{s}_{s,2}&a^{s}_{s,2}-a^{s}_{s,3}&\cdots&a^{s}_{s,s-1}-a^{s}_{s,s}&a^{s}_{s,s}
\end{matrix}\right|=\\
=(-1)^{s}\left|\begin{matrix}
a^{s-1}_{1,1}&a^{s-1}_{1,2}&\cdots&a^{s-1}_{1,s-1}\\
\vdots&\vdots&\ddots&\vdots\\
a^{s-1}_{s-1,1}&a^{s-1}_{s-1,2}&\cdots&a^{s-1}_{s-1,s-1}
\end{matrix}\right|
\end{gather*}
So this determinant is nonzero for any $s$, if and only if it is nonzero for $s=1$. But in this case the determinant is 1, since it is the determinant of matrix with one entry, concretely $a^{1}_{1,1}=1$.\par
We have constructed (not explicitly) the ambient operator inducing $P_{T}$. The same ambient operator will induce $\delta_{T}$, since it commutes with $\tilde{\Delta}$ and $r$. The only difference is in the weight of functions they are acting on, and hence in some coefficients depending on weight (action of $\mathbb{E}$ and $\bar{\mathbb{E}}$).
\end{proof}
\begin{theorem}
The vector space of symmetries of $\Delta$ modulo the equivalence relation is as a module for $SL(n+2,\mathbb{C})$ isomorphic to
\begin{displaymath}
	\bigoplus_{0\leq 2s\leq d}\quad\oooo{d-2s}{s}{s}{d-2s}\quad
\end{displaymath}
\end{theorem}
\begin{proof}
Since $\Delta$ is $SU(p+1,q+1)$-invariant operator acting on complex densities, its symmetries form a complex representation of $SU(p+1,q+1)$, and hence a representation of $SL(n+2,\mathbb{C})$. Proposition \ref{prop1} shows the existence of many symmetries of $\Delta$ on $M$. Together with Theorem \ref{thmsymb}, it also allows us to put any symmetry into a canonical form. Concretely, if $P$ is a symmetry operator of degree $(d,s)$, then we nay apply Theorem \ref{thmsymb} to normalize its symbol part $V^{a_{1}\dots a_{s}\sigma\dots\sigma}_{b_{1}\dots b_{s}}$ to lie in the kernel of corresponding BGG operator. Now consider $P-P_{V}$, where $P_{V}$ is from Proposition \ref{prop1}. By construction, this is a symmetry of $\Delta$ of degree less than $(d,s)$. Continuing this way we obtain a canonical form, namely
\begin{displaymath}
P_{V_{(d,s)}}+P_{V_{(d,s+1)}}+\dots+P_{V_{(1,0)}}+V_{0}
\end{displaymath} 
where $V_{(d',s')}$ lies in the kernel of the corresponding BGG operator ($V_{0}$ is a constant). Since the construction in the proof of Proposition \ref{prop1} is equivariant, Proposition \ref{prop1} together with Theorem \ref{thmsymb} imply that the vector space of symmetries of $\Delta$ is as a module for $SL(n+2,\mathbb{C})$ canonically isomorphic to
\begin{displaymath}
	\mathcal{A}=\bigoplus_{0\leq 2s\leq d}\quad\oooo{d-2s}{s}{s}{d-2s}\quad
\end{displaymath}
\end{proof}

\section{Decomposition of $S^{k}_{0}\mathfrak{sl}(V)$}

\subsection{$\End_{SL(V)}S^{k}_{0}\mathfrak{sl}(V)$}

We start with introducing some notation. Consider some complex vector space $V$ (of dimension $n$). Let ${e_{1},\dots,e_{n}}$ be some basis of $V$ and let ${\varepsilon^{1},\dots,\varepsilon^{n}}$ be the dual basis of $V^{*}$. There is a standard action of $GL(V)$ on $V$, $V^{*}$ and their tensor powers.\par
On $k$-th tensor power of $V$, we have a standard action of $\mathfrak{S}_{k}$. Let $I=(i_{1},\dots,i_{k})$ with $1\leq i_{1},\dots, i_{k}\leq n$ be a multiindex. Then the action of $s\in\mathfrak{S}_{k}$ on $\otimes^{k}V$ is given by
\begin{displaymath}
	s\cdot e_{I}\equiv e_{s\cdot I}:=e_{s^{-1}(i_{1})}\otimes\dots\otimes e_{s^{-1}(i_{k})}
\end{displaymath}
This means that if $s(i)=j$, then s takes the vector on $i$-th position and puts it on the $j$-th position. These operators are always linearly independent.\par
We try to do an analogy to Schur duality for $S^{k}_{0}\mathfrak{sl}(n+2,\mathbb{C})$. We start with computing $\End_{GL(V)}S^{k}\mathfrak{gl}(V)$.\par
We know that $\mathfrak{gl}(V)\cong V\otimes V^{*}$ as $GL(V)$-module. So we can identify $\End(\otimes^{k}\mathfrak{gl}(V))$ with $(\bigotimes^{2k}V)\otimes(\bigotimes^{2k}V^{*})$. Here we use the identification $(\bigotimes^{2k}V)\otimes(\bigotimes^{2k}V^{*})\cong\End((\bigotimes^{k}V)\otimes(\bigotimes^{k}V^{*}))$ given by
\begin{gather}
T=u^{1}\otimes\dots\otimes u^{2k}\otimes v_{1}\otimes\dots\otimes v_{2k}\mapsto\nonumber\\
\tilde{T}:w^{1}\otimes\dots\otimes w^{k}\otimes x_{1}\otimes\dots\otimes x_{k}\mapsto \nonumber\\
\prod_{i=1}^{k}u^{2i}(x_{i})\prod_{i=1}^{k}w^{i}(v_{2i})u^{1}\otimes u^{3}\otimes\dots\otimes u^{2k-1}\otimes v_{1}\otimes v_{3}\otimes\dots\otimes v_{2k-1}
\end{gather}
All $GL(V)$-invariant elements in $(\bigotimes^{2k}V)\otimes(\bigotimes^{2k}V^{*})$ are spanned by those of the form $C_{s}$ for some $s\in\mathfrak{S}_{2k}$ (see \cite{goodwall}), where
\begin{displaymath}
C_{s}:=\sum_{|I|=2k}\varepsilon^{I}\otimes e_{s\cdot I}
\end{displaymath}
Now we look at those operators, which preserve $S^{k}\mathfrak{gl}(V)$. For this purpose let's consider two subgroups of $\mathfrak{S}_{2k}$, denoted by $\mathfrak{S}^{1}_{k}$ and $\mathfrak{S}^{2}_{k}$. The group $\mathfrak{S}^{1}_{k}$ is the subgroup preserving all even numbers and $\mathfrak{S}^{2}_{k}$ is the subgroup preserving all odd numbers, respectively. To $\sigma\in\mathfrak{S}^{1}_{k}$ we associate a permutation $\hat{\sigma}\in\mathfrak{S}_{k}$ given by $\hat{\sigma}(i)=j$ if and only if $\sigma(2i-1)=2j-1$. Similarly, to $\sigma\in\mathfrak{S}^{2}_{k}$ we associate a permutation $\hat{\sigma}\in\mathfrak{S}_{k}$ given by $\hat{\sigma}(i)=j$ if and only if $\sigma(2i)=2j$.
\begin{lemma}\label{even}
For $\sigma\in\mathfrak{S}^{2}_{k}$ we have
\begin{displaymath}
C_{Ad_{\sigma}s}\cdot\varepsilon^{\hat{\sigma}\cdot I}\otimes e_{\hat{\sigma}\cdot J}=C_{s}\cdot\varepsilon^{I}\otimes e_{J}
\end{displaymath}
so the operators $C_{s}$ and $C_{Ad_{\sigma}s}$ induce the same operator on $S^{k}\mathfrak{gl}(V)$. Conversely, for given $\sigma\in\mathfrak{S}^{2}_{k}$, $C_{Ad_{\sigma}s}$ is the unique operator on $\otimes^{k}\mathfrak{gl}(V)$ mapping $\varepsilon^{\hat{\sigma}\cdot I}\otimes e_{\hat{\sigma}\cdot J}$ to 
$C_{s}\varepsilon^{I}\otimes e_{J}$ for any $(I,J)$.
\end{lemma}
\begin{proof}
First we observe that
\begin{gather}\label{ad}
C_{Ad_{\sigma}s}=\sum_{I}\varepsilon^{I}\otimes e_{\sigma s\sigma^{-1}\cdot I}=\sum_{J=\sigma^{-1}I}\varepsilon^{\sigma\cdot J}\otimes e_{\sigma s\cdot J}
\end{gather}
Then we have
\begin{gather*}
C_{Ad_{\sigma}s}\cdot\varepsilon^{\hat{\sigma}\cdot I}\otimes e_{\hat{\sigma}\cdot J}=\\
=\sum_{|L|=2k}\prod_{m=1}^{k}\varepsilon^{l_{\sigma^{-1}(2m)}}(e_{j_{\hat{\sigma}^{-1}}(m)})\prod_{m=1}^{k}\varepsilon^{i_{\hat{\sigma}^{-1}(m)}}(e_{l_{s^{-1}\sigma^{-1}(2m)}})\cdot\\
\cdot\varepsilon^{l_{1}}\otimes\varepsilon^{l_{3}}\otimes\dots\otimes\varepsilon^{l_{2k-1}}\otimes e_{l_{s^{-1}(1)}}\otimes e_{l_{s^{-1}(3)}}\otimes\dots\otimes e_{l_{s^{-1}(2k-1)}}=\\
=\sum_{|L|=2k}\prod_{m=1}^{k}\varepsilon^{l_{2m}}(e_{j_{m}})\prod_{m=1}^{k}\varepsilon^{i_{m}}(e_{l_{2m}})\cdot\\
\cdot\varepsilon^{l_{1}}\otimes\varepsilon^{l_{3}}\otimes\dots\otimes\varepsilon^{l_{2k-1}}\otimes e_{l_{s^{-1}(1)}}\otimes e_{l_{s^{-1}(3)}}\otimes\dots\otimes e_{l_{s^{-1}(2k-1)}}=\\
=C_{s}\cdot\varepsilon^{I}\otimes e_{J}
\end{gather*}
Since $S^{k}\mathfrak{gl}(V)$ is generated by elements of the form $\sum_{\sigma\in\mathfrak{S}_{k}}\varepsilon^{\sigma\cdot I}\otimes e_{\sigma\cdot J}$, we see, that $C_{s}$ and $C_{Ad_{\sigma}s}$ really induce the same operator on $S^{k}\mathfrak{gl}(V)$.\par
For the converse, we shall realize that $\{\varepsilon^{I}\otimes e_{J}\}_{I,J}$ is a basis of $\otimes^{k}\mathfrak{gl}(V)$ and every operator is given by its values on the basis.
\end{proof}
\begin{lemma}\label{odd}
Let $\sigma\in\mathfrak{S}^{1}_{k}$. Then if 
\begin{displaymath}
C_{s}\cdot\varepsilon^{I}\otimes e_{J}=\sum_{P,Q}A^{Q}_{P}\varepsilon^{P}\otimes e_{Q}
\end{displaymath}
we have
\begin{displaymath}
C_{Ad_{\sigma}s}\cdot\varepsilon^{I}\otimes e_{J}=\sum_{P,Q}A^{Q}_{P}\varepsilon^{\hat{\sigma}\cdot P}\otimes e_{\hat{\sigma}\cdot Q}
\end{displaymath}
Conversely, $C_{Ad_{\sigma}s}$ is the only operator with the above property.
\end{lemma}
\begin{proof}
Put $p_{i}:=l_{2i-1}$ and $q_{i}:=l_{s^{-1}(2i-1)}$. Then we have
\begin{gather*}
C_{Ad_{\sigma}s}\cdot\varepsilon^{I}\otimes e_{J}=\sum_{|L'|=2k}\prod_{m=1}^{k}\varepsilon^{l'_{2m}}(e_{j_{m}})\prod_{m=1}^{k}\varepsilon^{i_{m}}(e_{l'_{\sigma s^{-1}(2m)}})\cdot\\
\cdot\varepsilon^{l'_{1}}\otimes\varepsilon^{l'_{3}}\otimes\dots\otimes\varepsilon^{l'_{2k-1}}\otimes
e_{l'_{\sigma s^{-1}\sigma^{-1}(1)}}\otimes e_{l'_{\sigma s^{-1}\sigma^{-1}(3)}}\otimes\dots\otimes e_{l'_{\sigma s^{-1}\sigma^{-1}(2k-1)}}=\\
=\sum_{L'=\sigma\cdot L}\prod_{m=1}^{k}\varepsilon^{l_{2m}}(e_{j_{m}})\prod_{m=1}^{k}\varepsilon^{i_{m}}(e_{l_{s^{-1}(2m)}})\cdot\\
\cdot\varepsilon^{l_{\sigma^{-1}(1)}}\otimes\varepsilon^{l_{\sigma^{-1}(3)}}\otimes\dots\otimes\varepsilon^{l_{\sigma^{-1}(2k-1)}}\otimes
e_{l_{s^{-1}\sigma^{-1}(1)}}\otimes e_{l_{s^{-1}\sigma^{-1}(3)}}\otimes\dots\otimes e_{l_{s^{-1}\sigma^{-1}(2k-1)}}=\\
=\sum_{|L|=2k}\prod_{m=1}^{k}\varepsilon^{l_{2m}}(e_{j_{m}})\prod_{m=1}^{k}\varepsilon^{i_{m}}(e_{l_{s^{-1}(2m)}})
\varepsilon^{\hat{\sigma}\cdot P}\otimes e_{\hat{\sigma}\cdot Q}
\end{gather*}
By (\ref{ad}) we have
\begin{gather*}
C_{s}\cdot\varepsilon^{I}\otimes e_{J}=\sum_{|L|=2k}\prod_{m=1}^{k}\varepsilon^{l_{2m}}(e_{j_{m}})\prod_{m=1}^{k}\varepsilon^{i_{m}}(e_{l_{s^{-1}(2m)}})\varepsilon^{P}\otimes e_{Q}
\end{gather*}
The proof of the converse is the same as for the previous lemma.
\end{proof}
\begin{theorem}
The algebra $\End_{GL(V)}S^{k}\mathfrak{gl}(V)$ is generated by operators of the form
\begin{displaymath}
\sum_{\sigma\in\mathfrak{S}^{1}_{k}\times\mathfrak{S}^{2}_{k}}\frac{1}{(k!)^{2}}C_{Ad_{\sigma}s}
\end{displaymath}
\end{theorem}
\begin{proof}
This is a direct consequence of Lemma \ref{even} and Lemma \ref{odd}. Since all $C_{s}$ are linearly independent, it suffices for given $s$ to symmetrize the result of $C_{s}\cdot\varepsilon^{I}\otimes e_{J}$. From Lemma \ref{odd} we know one way how to do it. The linear independence of operators $C_{s}$ tells us that we can't replace $C_{Ad_{\sigma}s}$ by any expression of the form $\sum_{s'\neq Ad_{\sigma}s}a_{s'}C_{s'}$. So $\sum_{\sigma\in\mathfrak{S}^{1}_{k}}C_{Ad_{\sigma}s}$ is not only the only operator symmetrizing the result of $C_{s}$ (up to constant multiple), but also the only possible expression of this operator (provided that all operators in the expression are distinct). Lemma \ref{even} tells us that we can anywhere in this expression replace $s$ by $Ad_{\sigma}s$, $s\in\mathfrak{S}^{2}_{k}$, when acting on symmetric tensors.
\end{proof}
It remains to investigate the behaviour of the operators $C_{s}$ with respect to traces. Looking at the action of $C_{s}$ on general tensors, we have
\begin{gather*}
C_{s}\cdot\sum_{I,J}v^{J}_{I}\varepsilon^{I}\otimes e_{J}=\sum_{I,J}v^{J}_{I}\sum_{|L|=2k}\prod_{m=1}^{k}\delta^{l_{2m}}_{j_{m}}\prod_{m=1}^{k}\delta^{i_{m}}_{l_{s^{-1}(2m)}}\cdot\\
\cdot\varepsilon^{l_{1}}\otimes\varepsilon^{l_{3}}\otimes\dots\otimes\varepsilon^{l_{2k-1}}\otimes e_{l_{s^{-1}(1)}}\otimes e_{l_{s^{-1}(3)}}\otimes\dots
\otimes e_{l_{s^{-1}(2k-1)}}
\end{gather*}
Assume there is a pair $(m,n)$ such that $s(2m)=2n$. This poses a condition on possible pairs $(I,J)$, namely $i_{n}=j_{m}$. Since we are summing over all possible pairs $(I,J)$, we see that in the resulting expression we are summing not $v^{I}_{J}$-s, but their traces over $i_{n}=j_{m}$. This means that for trace-free tensors the result is always zero.\par
Assume now there is no such pair. In this case for every $m$ there exists $n$ such that $s(2m-1)=2n$. So we can associate to $s$ two permutations $\sigma_{1},\sigma_{2}\in\mathfrak{S}_{k}$ given by $\sigma^{s}_{1}(m)=n$ if and only if $s(2m-1)=2n$ and $\sigma^{s}_{2}(m)=n$ if and only if $s(2m)=2n-1$, respectively. In addition, the index $L$ is fully determined by $(I,J)$. This property is invariant under $Ad(\mathfrak{S}^{1}_{k}\times\mathfrak{S}^{2}_{k})$. Putting all this together, we get
\begin{gather}\label{sigma_{1,2}}
C_{s}\cdot\sum_{I,J}v^{J}_{I}\varepsilon^{I}\otimes e_{J}=\sum_{I,J}v^{J}_{I}\varepsilon^{(\sigma^{s}_{1})^{-1}\cdot I}\otimes e_{\sigma^{s}_{2}\cdot J}
\end{gather}
so it is nonzero on $S^{k}_{0}\mathfrak{gl}(V)$.\par
Since $S^{k}_{0}\mathfrak{gl}(V)=S^{k}_{0}\mathfrak{sl}(V)$, and the adjoint action of centre of $GL(V)$ on $\mathfrak{gl}(V)$ is trivial by definition, we have proved the following
\begin{theorem}
The algebra $\End_{SL(V)}S^{k}_{0}\mathfrak{sl}(V)$ is (as vector space) generated by operators of the form
\begin{gather}\label{basis}
\sum_{\sigma\in\mathfrak{S}^{1}_{k}\times\mathfrak{S}^{2}_{k}}\frac{1}{(k!)^{2}}C_{Ad_{\sigma}s}
\end{gather}
where $s\in\mathfrak{S}_{2k}$ is such that it interchanges even and odd numbers.
\end{theorem}
\begin{example}
For $k=1$ there is only one permutation exchanging odd and even numbers, namely transposition $(12)$. The induced operator is simply the identity.
\end{example}
\begin{example}
For $k=2$ there are two classes of permutations exchanging odd and even numbers. First class consists of
\begin{displaymath}
s_{1}=(12)(34)\qquad s_{2}=(14)(32)	
\end{displaymath}
with the induced operator being the identity, and the second one consists of
\begin{displaymath}
s_{3}=(1234)\qquad s_{4}=(1432)	
\end{displaymath}
with the induced operator being
\begin{gather}\label{2transposition}
\frac{1}{4}(2C_{s_{3}}+2C_{s_{4}})\cdot\varepsilon^{I}\otimes e_{J}=\frac{1}{2}(\varepsilon^{i_{1}}\otimes\varepsilon^{i_{2}}\otimes e_{j_{2}}\otimes e_{j_{1}}+
\varepsilon^{i_{2}}\otimes\varepsilon^{i_{1}}\otimes e_{j_{1}}\otimes e_{j_{2}}),
\end{gather}
respectively.
\end{example}
\begin{example}\label{k3}
For $k=3$ there are three classes of permutations interchanging odd and even numbers. First class contains
\begin{gather*}
s_{1}=(12)(34)(56)\qquad s_{3}=(14)(32)(56)\qquad s_{5}=(16)(32)(54)\\
s_{2}=(12)(36)(54)\qquad s_{4}=(14)(36)(52)\qquad s_{6}=(16)(34)(52)
\end{gather*}
with the induced operator being the identity. The second class contains
\begin{gather*}
s_{7}=(12)(3456)\qquad s_{9}=(14)(3256)\qquad s_{11}=(16)(3254)\\
s_{8}=(12)(3654)\qquad s_{10}=(14)(3652)\qquad s_{12}=(16)(3452)\\
s_{13}=(32)(1456)\qquad s_{15}=(34)(1256)\qquad s_{17}=(36)(1254)\\
s_{14}=(32)(1654)\qquad s_{16}=(34)(1652)\qquad s_{18}=(36)(1452)\\
s_{19}=(52)(1436)\qquad s_{21}=(54)(1236)\qquad s_{23}=(56)(1234)\\
s_{20}=(52)(1634)\qquad s_{22}=(54)(1632)\qquad s_{24}=(56)(1432)
\end{gather*}
with the induced operator being
\begin{gather}\label{3transposition}
\frac{1}{36}\left(2C_{s_{7}}+2C_{s_{8}}+2C_{s_{9}}+2C_{s_{10}}+2C_{s_{11}}+2C_{s_{12}}+\right.\nonumber\\
\left. +2C_{s_{13}}+2C_{s_{14}}+2C_{s_{15}}+2C_{s_{16}}+2C_{s_{17}}+2C_{s_{18}}+\right.\nonumber\\
\left. +2C_{s_{19}}+2C_{s_{20}}+2C_{s_{21}}+2C_{s_{22}}+2C_{s_{23}}+2C_{s_{24}}\right)\cdot\varepsilon^{I}\otimes e_{J}=\nonumber\\
=\frac{1}{18}\left(\varepsilon^{i_{1}}\otimes\varepsilon^{i_{2}}\otimes\varepsilon^{i_{3}}\otimes e_{j_{1}}\otimes e_{j_{3}}\otimes e_{j_{2}}+
\varepsilon^{i_{1}}\otimes\varepsilon^{i_{3}}\otimes\varepsilon^{i_{2}}\otimes e_{j_{1}}\otimes e_{j_{2}}\otimes e_{j_{3}}+\right.\nonumber\\
\left. +\varepsilon^{i_{2}}\otimes\varepsilon^{i_{1}}\otimes\varepsilon^{i_{3}}\otimes e_{j_{2}}\otimes e_{j_{3}}\otimes e_{j_{1}}+
\varepsilon^{i_{2}}\otimes\varepsilon^{i_{3}}\otimes\varepsilon^{i_{1}}\otimes e_{j_{2}}\otimes e_{j_{1}}\otimes e_{j_{3}}+\right.\\
\left. +\varepsilon^{i_{3}}\otimes\varepsilon^{i_{1}}\otimes\varepsilon^{i_{2}}\otimes e_{j_{3}}\otimes e_{j_{2}}\otimes e_{j_{1}}+
\varepsilon^{i_{3}}\otimes\varepsilon^{i_{2}}\otimes\varepsilon^{i_{1}}\otimes e_{j_{3}}\otimes e_{j_{1}}\otimes e_{j_{2}}+\right.\nonumber\\
\left. +\varepsilon^{i_{2}}\otimes\varepsilon^{i_{1}}\otimes\varepsilon^{i_{3}}\otimes e_{j_{3}}\otimes e_{j_{1}}\otimes e_{j_{2}}+
\varepsilon^{i_{3}}\otimes\varepsilon^{i_{1}}\otimes\varepsilon^{i_{2}}\otimes e_{j_{2}}\otimes e_{j_{1}}\otimes e_{j_{3}}+\right.\nonumber\\
\left. +\varepsilon^{i_{1}}\otimes\varepsilon^{i_{2}}\otimes\varepsilon^{i_{3}}\otimes e_{j_{3}}\otimes e_{j_{2}}\otimes e_{j_{1}}+
\varepsilon^{i_{3}}\otimes\varepsilon^{i_{2}}\otimes\varepsilon^{i_{1}}\otimes e_{j_{1}}\otimes e_{j_{2}}\otimes e_{j_{6}}+\right.\nonumber\\
\left. +\varepsilon^{i_{1}}\otimes\varepsilon^{i_{3}}\otimes\varepsilon^{i_{2}}\otimes e_{j_{2}}\otimes e_{j_{3}}\otimes e_{j_{1}}+
\varepsilon^{i_{2}}\otimes\varepsilon^{i_{3}}\otimes\varepsilon^{i_{1}}\otimes e_{j_{1}}\otimes e_{j_{3}}\otimes e_{j_{2}}+\right.\nonumber\\
\left. +\varepsilon^{i_{2}}\otimes\varepsilon^{i_{3}}\otimes\varepsilon^{i_{1}}\otimes e_{j_{3}}\otimes e_{j_{2}}\otimes e_{j_{1}}+
\varepsilon^{i_{3}}\otimes\varepsilon^{i_{2}}\otimes\varepsilon^{i_{1}}\otimes e_{j_{2}}\otimes e_{j_{3}}\otimes e_{j_{1}}+\right.\nonumber\\
\left. +\varepsilon^{i_{1}}\otimes\varepsilon^{i_{3}}\otimes\varepsilon^{i_{2}}\otimes e_{j_{3}}\otimes e_{j_{1}}\otimes e_{j_{2}}+
\varepsilon^{i_{3}}\otimes\varepsilon^{i_{1}}\otimes\varepsilon^{i_{2}}\otimes e_{j_{1}}\otimes e_{j_{3}}\otimes e_{j_{2}}+\right.\nonumber\\
\left. +\varepsilon^{i_{1}}\otimes\varepsilon^{i_{2}}\otimes\varepsilon^{i_{3}}\otimes e_{j_{2}}\otimes e_{j_{1}}\otimes e_{j_{3}}+
\varepsilon^{i_{2}}\otimes\varepsilon^{i_{1}}\otimes\varepsilon^{i_{3}}\otimes e_{j_{1}}\otimes e_{j_{2}}\otimes e_{j_{3}}\right)\nonumber
\end{gather}
The third class contains
\begin{gather*}
s_{25}=(123456)\qquad s_{27}=(143256)\qquad s_{29}=(163254)\\
s_{26}=(123654)\qquad s_{28}=(143652)\qquad s_{30}=(163452)\\
s_{31}=(125436)\qquad s_{33}=(145236)\qquad s_{35}=(165234)\\
s_{32}=(125634)\qquad s_{34}=(145632)\qquad s_{36}=(165432)
\end{gather*}
with the induced operator being
\begin{gather}\label{3-cycle}
\frac{1}{36}\left(3C_{s_{25}}+3C_{s_{26}}+3C_{s_{27}}+3C_{s_{28}}+3C_{s_{29}}+3C_{s_{30}}+\right.\nonumber\\
\left. +3C_{s_{31}}+3C_{s_{32}}+3C_{s_{33}}+3C_{s_{34}}+3C_{s_{35}}+3C_{s_{36}}\right)\cdot\varepsilon^{I}\otimes e_{J}=\nonumber\\
=\frac{1}{12}\left(\varepsilon^{i_{1}}\otimes\varepsilon^{i_{2}}\otimes\varepsilon^{i_{3}}\otimes e_{j_{3}}\otimes e_{j_{1}}\otimes e_{j_{2}}+
\varepsilon^{i_{1}}\otimes\varepsilon^{i_{3}}\otimes\varepsilon^{i_{2}}\otimes e_{j_{2}}\otimes e_{j_{1}}\otimes e_{j_{3}}+\right.\nonumber\\
\left. +\varepsilon^{i_{2}}\otimes\varepsilon^{i_{1}}\otimes\varepsilon^{i_{3}}\otimes e_{j_{3}}\otimes e_{j_{2}}\otimes e_{j_{1}}+
\varepsilon^{i_{2}}\otimes\varepsilon^{i_{3}}\otimes\varepsilon^{i_{1}}\otimes e_{j_{1}}\otimes e_{j_{2}}\otimes e_{j_{3}}+\right.\\
\left. +\varepsilon^{i_{3}}\otimes\varepsilon^{i_{1}}\otimes\varepsilon^{i_{2}}\otimes e_{j_{2}}\otimes e_{j_{3}}\otimes e_{j_{1}}+
\varepsilon^{i_{3}}\otimes\varepsilon^{i_{2}}\otimes\varepsilon^{i_{1}}\otimes e_{j_{1}}\otimes e_{j_{3}}\otimes e_{j_{2}}+\right.\nonumber\\
\left. +\varepsilon^{i_{1}}\otimes\varepsilon^{i_{3}}\otimes\varepsilon^{i_{2}}\otimes e_{j_{3}}\otimes e_{j_{2}}\otimes e_{j_{1}}+
\varepsilon^{i_{1}}\otimes\varepsilon^{i_{2}}\otimes\varepsilon^{i_{3}}\otimes e_{j_{2}}\otimes e_{j_{3}}\otimes e_{j_{1}}+\right.'\nonumber\\
\left. +\varepsilon^{i_{2}}\otimes\varepsilon^{i_{3}}\otimes\varepsilon^{i_{1}}\otimes e_{j_{3}}\otimes e_{j_{1}}\otimes e_{j_{2}}+
\varepsilon^{i_{2}}\otimes\varepsilon^{i_{1}}\otimes\varepsilon^{i_{3}}\otimes e_{j_{1}}\otimes e_{j_{3}}\otimes e_{j_{2}}+\right.\nonumber\\
\left. +\varepsilon^{i_{3}}\otimes\varepsilon^{i_{2}}\otimes\varepsilon^{i_{1}}\otimes e_{j_{2}}\otimes e_{j_{1}}\otimes e_{j_{3}}+
\varepsilon^{i_{3}}\otimes\varepsilon^{i_{1}}\otimes\varepsilon^{i_{2}}\otimes e_{j_{1}}\otimes e_{j_{2}}\otimes e_{j_{3}}\right)\nonumber
\end{gather}
\end{example}

\subsection{Multiplicative structure of $\End_{SL(V)}S^{k}_{0}\mathfrak{sl}(V)$}

We start with determining the algebra structure of $\End_{SL(V)}S^{k}_{0}\mathfrak{sl}(V)$. We know that to every $s\in\mathfrak{S}_{2k}$ interchanging odd and even numbers we can associate two permutations $\sigma^{s}_{1}, \sigma^{s}_{2}\in\mathfrak{S}_{k}$. Looking at the expression (\ref{sigma_{1,2}}) we see that the permutation $\tilde{\sigma}^{s}=\sigma^{s}_{2}\sigma^{s}_{1}$ represents something like 'relative position' of indices $I$ and $J$ in the resulting expression. This means that the resulting expression is of the form
\begin{gather}\label{s}
\varepsilon^{i_{l^{s}_{1}}}\otimes\varepsilon^{i_{l^{s}_{2}}}\otimes\dots\otimes\varepsilon^{i_{l^{s}_{k}}}\otimes e_{j_{l^{s}_{(\tilde{\sigma}^{s})^{-1}(1)}}}\otimes e_{j_{l^{s}_{(\tilde{\sigma}^{s})^{-1}(2)}}}\otimes\dots\otimes e_{j_{l^{s}_{(\tilde{\sigma}^{s})^{-1}(k)}}}
\end{gather}
where we can view $l$ as $l\in\mathfrak{S}_{k}$: $l(i):=l_{i}$.\par
Conjugating $s$ by $\sigma\in\mathfrak{S}^{2}_{k}$, we have $C_{Ad_{\sigma}s}\cdot\varepsilon^{I}\otimes e_{J}=C_{s}\cdot\varepsilon^{\hat{\sigma}^{-1}\cdot I}\otimes e_{\hat{\sigma}^{-1}\cdot J}$. Comparing upper indices we get $l^{Ad_{\sigma}s}=l^{s}\circ\hat{\sigma}$. Since $\tilde{\sigma}^{s}$ only depends on $s$ and not on the tensors acted on, we have that $\tilde{\sigma}^{Ad_{\sigma}s}=\tilde{\sigma}^{s}$. Conjugating $s$ by $\sigma\in\mathfrak{S}^{1}_{k}$, we get \begin{gather}\label{Ad_{sigma}s}
C_{Ad_{\sigma}s}\cdot\varepsilon^{I}\otimes e_{J}=\varepsilon^{l^{s}_{\hat{\sigma}^{-1}(1)}}\otimes\dots\otimes\varepsilon^{l^{s}_{\hat{\sigma}^{-1}(k)}}\otimes
e_{j_{l^{s}_{(\tilde{\sigma}^{s})^{-1}\hat{\sigma}^{-1}(1)}}}\otimes\dots\otimes e_{j_{l^{s}_{(\tilde{\sigma}^{s})^{-1}\hat{\sigma}^{-1}(k)}}}
\end{gather}
Comparing upper indices, we see that $l^{Ad_{\sigma}s}=l^{s}\circ\hat{\sigma}^{-1}$. Comparing lower indices, we get $\tilde{\sigma}^{Ad_{\sigma}s}=Ad_{\hat{\sigma}}\tilde{\sigma}^{s}$.
\begin{proposition}
Let's consider $s,s'\in\mathfrak{S}_{2k}$ interchanging odd and even numbers. Then there exists $\sigma\in\mathfrak{S}^{1}_{k}\times\mathfrak{S}^{2}_{k}$, s.t. $s'=Ad_{\sigma}s$ if and only if there exists $\sigma'\in\mathfrak{S}_{k}$, s.t. $\tilde{\sigma}^{s'}=Ad_{\sigma'}\tilde{\sigma}^{s}$.
\end{proposition}
\begin{proof}
The 'only if' part is proved above. For the 'if' part, we shall realize that any $s$ interchanging odd and even numbers is fully determined by $l^{s}$ and $\tilde{\sigma}^{s}$ because of formula (\ref{s}). Let's have given $s,s'\in\mathfrak{S}_{2k}$ with a $\sigma\in\mathfrak{S}_{k}$, s.t. $\tilde{\sigma}^{s'}=Ad_{\sigma}\tilde{\sigma}^{s}$. Then there is a $\sigma'\in\mathfrak{S}^{1}_{k}$ with $\hat{\sigma'}=\sigma$. Put $s'':=Ad_{\sigma'}s$. Then we have $\tilde{\sigma}^{s'}=\tilde{\sigma}^{s''}$. Since $l^{s'}$ and $l^{s''}$ are both permutations, there surely exists a permutation $\sigma''\in\mathfrak{S}_{k}$ with $l^{s'}=l^{s''}\circ\sigma''$. To $\sigma''$ we can associate a unique permutation $\sigma'''\in\mathfrak{S}^{2}_{k}$ with $\hat{\sigma'''}=\sigma''$. Putting $s'''=Ad_{\sigma'''}s''$, we see that $l^{s'}=l^{s'''}$ and $\tilde{\sigma}^{s'}=\tilde{\sigma}^{s'''}$, i.e. $s'=s'''$.
\end{proof}
This proposition says that there is a well-defined injective map mapping a basis element (\ref{basis}) of $\End_{SL(V)}S^{k}_{0}\mathfrak{sl}(V)$ to some conjugacy class of $S_{k}$. To see that this map is onto, it suffices for any $\sigma\in\mathfrak{S}^{k}$ find some $s\in\mathfrak{S}_{2k}$ interchanging odd and even numbers with $\tilde{\sigma}^{s}=\sigma$. For simplicity, let's assume that $\sigma_{1}^{s}=Id$. In this case $\tilde{\sigma}^{s}=\sigma^{s}_{2}$. But $s$ is determined uniquely by $\sigma^{s}_{1}$ and $\sigma^{s}_{2}$, since $\sigma^{1}_{s}$ describes the restriction of $s$ to odd numbers, and $\sigma^{s}_{2}$ determines the restriction of $s$ to even numbers. So we see we have a one-to-one correspondence between the basis elements (\ref{basis}) of $\End_{SL(V)}S^{k}_{0}\mathfrak{sl}(V)$ and conjugacy classes of $S_{k}$.
\begin{example}
The case $k=1$ is trivial, since there is only one basis element (the identity) and only one conjugacy class in $S_{1}$ - that of identity.
\end{example}
\begin{example}
In the case $k=2$ we have two basis elements. The first one is the identity. In this case $\tilde{\sigma}$ is also the identity. To the second one we can associate a transposition, as can be easily seen from (\ref{2transposition}).
\end{example}
\begin{example}
In the case $k=3$ there are three basis elements. We treat them in the same order as in Example \ref{k3}. The first one is the identity with $\tilde{\sigma}$ being the identity. The second one corresponds to transposition, as can be seen from (\ref{3transposition}), and the third one corresponds to three-cycle, as can be seen from (\ref{3-cycle}).
\end{example}
For notational reasons let's define $s'\cdot s$ by $C_{s'\cdot s}=C_{s'}\cdot C_{s}$ and for $\sigma\in\mathfrak{S}_{k}$ let $C_{(\sigma)}$ be the basis element of $\End_{SL(V)}S^{k}\mathfrak{sl}(V)$ corresponding to the conjugacy class of $\sigma$, which we will denote by $(\sigma)$.
\begin{lemma}
Let $s,s'\in\mathfrak{S}_{2k}$ interchange odd and even numbers. Then $\tilde{\sigma}^{s'\cdot s}=\tilde{\sigma}^{s'}Ad_{(\sigma^{s'}_{1})^{-1}}\tilde{\sigma}^{s}$
\end{lemma}
\begin{proof}
We have
\begin{gather*}
C_{s'}\cdot C_{s}\cdot\varepsilon^{I}\otimes e_{J}=C_{s'}\cdot\varepsilon^{(\sigma^{s}_{1})^{-1}\cdot I}\otimes e_{\sigma^{s}_{2}\cdot J}=\varepsilon^{(\sigma^{s'}_{1})^{-1}(\sigma^{s}_{1})^{-1}\cdot I}\otimes e_{\sigma^{s'}_{2}\sigma^{s}_{2}\cdot J}
\end{gather*}
so we have
\begin{gather*}
\tilde{\sigma}^{s'\cdot s}=\sigma^{s'}_{2}\sigma^{s}_{2}\sigma^{s}_{1}\sigma^{s'}_{1}=\sigma^{s'}_{2}\sigma^{s'}_{1}(\sigma^{s'}_{1})^{-1}\sigma^{s}_{2}\sigma^{s}_{1}\sigma^{s'}_{1}
=\tilde{\sigma}^{s'}(\sigma^{s'}_{1})^{-1}\tilde{\sigma}^{s}\sigma^{s'}_{1}
\end{gather*}
\end{proof}
\begin{theorem}
Assume $\sigma,\sigma'\in\mathfrak{S}_{k}$. Then $C_{(\sigma')}\cdot C_{(\sigma)}=\sum_{(\tau)\subset\mathfrak{S}_{k}}A_{(\tau)}C_{(\tau)}$, where $A_{(\tau)}$ is the probability of $\hat{\sigma}'\hat{\sigma}\in(\tau)$ for $\hat{\sigma}'\in(\sigma')$ and $\hat{\sigma}\in\sigma$.
\end{theorem}
\begin{proof}
 We will compute in the group algebra $\mathbb{C}[\mathfrak{S}_{k}]$. First observe that for given $s$ and $s'$ we do not need to know $C_{s'\cdot s}$, but it suffices to know $\tilde{\sigma}^{s'\cdot s}$. Now
\begin{gather*}
\sum_{\sigma\in\mathfrak{S}^{1}_{k}\times\mathfrak{S}^{2}_{k}}\sum_{\sigma'\in\mathfrak{S}^{1}_{k}\times\mathfrak{S}^{2}_{k}}\frac{1}{(k!)^{4}}C_{Ad_{\sigma'}s'}\cdot C_{Ad_{\sigma}s}=\\
=\sum_{\sigma'\in\mathfrak{S}^{1}_{k}\times\mathfrak{S}^{2}_{k}}\frac{1}{(k!)^{2}}C_{Ad_{\sigma'}s'}\cdot\sum_{\sigma\in\mathfrak{S}^{1}_{k}\times\mathfrak{S}^{2}_{k}}\frac{1}{(k!)^{2}}C_{Ad_{\sigma}s}
\end{gather*}
We start with $\sigma'=Id$ and try to compute
\begin{displaymath}
C_{s'}\cdot\sum_{\sigma\in\mathfrak{S}^{1}_{k}\times\mathfrak{S}^{2}_{k}}\frac{1}{(k!)^{2}}C_{Ad_{\sigma}s}
\end{displaymath}
It suffices to compute for each $\sigma\in\mathfrak{S}^{1}_{k}\times\mathfrak{S}^{2}_{k}$ corresponding $\tilde{\sigma}^{s'\cdot Ad_{\sigma}s}$. For this we have
\begin{gather*}
\frac{1}{(k!)^{2}}\sum_{\sigma\in\mathfrak{S}^{1}_{k}\times\mathfrak{S}^{2}_{k}}\tilde{\sigma}^{s'\cdot Ad_{\sigma}s}=\frac{1}{(k!)^{2}}\sum_{\sigma\in\mathfrak{S}^{1}_{k}\times\mathfrak{S}^{2}_{k}}\tilde{\sigma}^{s'}Ad_{(\sigma^{s'}_{1})^{-1}}\tilde{\sigma}^{Ad_{\sigma}s}
\end{gather*}
Now we associate to $\sigma$ a permutation $\sigma''\in\mathfrak{S}^{1}_{k}$ by $\sigma''(2i-1)=\sigma(2i-1)$. Then by Lemma \ref{even} we have $\tilde{\sigma}^{Ad_{\sigma}s}=\tilde{\sigma}^{Ad_{\sigma''}s}$, so
\begin{gather*}
\frac{1}{(k!)^{2}}\sum_{\sigma\in\mathfrak{S}^{1}_{k}\times\mathfrak{S}^{2}_{k}}\tilde{\sigma}^{s'}Ad_{(\sigma^{s'}_{1})^{-1}}\tilde{\sigma}^{Ad_{\sigma}s}=
\frac{1}{(k!)^{2}}\sum_{\sigma\in\mathfrak{S}^{1}_{k}\times\mathfrak{S}^{2}_{k}}\tilde{\sigma}^{s'}Ad_{(\sigma^{s'}_{1})^{-1}}\tilde{\sigma}^{Ad_{\sigma''}s}=\\
=\frac{1}{k!}\sum_{\sigma''\in\mathfrak{S}^{1}_{k}}\tilde{\sigma}^{s'}Ad_{(\sigma^{s'}_{1})^{-1}}\tilde{\sigma}^{Ad_{\sigma''}s}=
\frac{1}{k!}\sum_{\sigma''\in\mathfrak{S}^{1}_{k}}\tilde{\sigma}^{s'}Ad_{(\sigma^{s'}_{1})^{-1}}Ad_{(\widehat{\sigma''})^{-1}}\tilde{\sigma}^{s}=\\
=\frac{1}{k!}\sum_{\widehat{\sigma''}\in\mathfrak{S}_{k}}\tilde{\sigma}^{s'}Ad_{(\sigma^{s'}_{1})^{-1}}Ad_{(\widehat{\sigma''})^{-1}}\tilde{\sigma}^{s}=
\frac{1}{k!}\sum_{\widehat{\sigma''}\in\mathfrak{S}_{k}}\tilde{\sigma}^{s'}Ad_{(\widehat{\sigma''})^{-1}}\tilde{\sigma}^{s}
\end{gather*}
So we see that 
\begin{gather*}
\sum_{\sigma}\tilde{\sigma}^{s'\cdot Ad_{\sigma}s}=\frac{1}{|(\tilde{\sigma}^{s})|}\sum_{\sigma\in(\tilde{\sigma}^{s})}\tilde{\sigma}^{s'}\sigma
\end{gather*}
Now
\begin{gather*}
\frac{1}{(k!)^{4}}\sum_{\sigma'\in\mathfrak{S}^{1}_{k}\times\mathfrak{S}^{2}_{k}}\sum_{\sigma\in\mathfrak{S}^{1}_{k}\times\mathfrak{S}^{2}_{k}}\tilde{\sigma}^{Ad_{\sigma'}s'\cdot Ad_{\sigma}s}=\frac{1}{(k!)^{2}}\frac{1}{|(\tilde{\sigma}^{s})|}\sum_{\sigma'\in\mathfrak{S}^{1}_{k}\times\mathfrak{S}^{2}_{k}}
\sum_{\sigma\in(\tilde{\sigma})}\tilde{\sigma}^{Ad_{\sigma'}s'}\sigma
\end{gather*}
We associate $\sigma'''$ to $\sigma'$ in the same way as we have associated $\sigma''$ to $\sigma$. So we get
\begin{gather*}
\frac{1}{(k!)^{2}}\frac{1}{|(\tilde{\sigma}^{s})|}\sum_{\sigma'\in\mathfrak{S}^{1}_{k}\times\mathfrak{S}^{2}_{k}}\sum_{\sigma\in(\tilde{\sigma})}\tilde{\sigma}^{Ad_{\sigma'}s'}\sigma=\frac{1}{k!}\frac{1}{|(\tilde{\sigma}^{s})|}\sum_{\sigma'''\in\mathfrak{S}^{1}_{k}}\sum_{\sigma\in(\tilde{\sigma})}\tilde{\sigma}^{Ad_{\sigma'''}s'}\sigma=\\
=\frac{1}{k!}\frac{1}{|(\tilde{\sigma}^{s})|}\sum_{\widehat{\sigma'''}\in\mathfrak{S}_{k}}\sum_{\sigma\in(\tilde{\sigma})}Ad_{(\widehat{\sigma'''})^{-1}}\tilde{\sigma}^{s'}\sigma=\frac{1}{|(\tilde{\sigma}^{s'})|}\frac{1}{|(\tilde{\sigma}^{s})|}\sum_{\sigma'\in(\tilde{\sigma}^{s'})}\sum_{\sigma\in(\tilde{\sigma}^{s})}\sigma'\sigma
\end{gather*}
The only interesting thing on these terms is their conjugacy class. If we replace the terms with the same conjugacy class by its distinguished representative, the coefficients $A(\tau)$ at the representative of some $(\tau)$ will express the probability of $\sigma'\sigma\in(\tau)$ for $\sigma'\in(\tilde{\sigma}^{s'})$ and $\sigma\in(\tilde{\sigma}^{s})$. Moreover, the result has the form
\begin{gather*}
\sum_{(\tau)\subset\mathfrak{S}_{k}}A_{(\tau)}\frac{1}{|(\tau)|}\sum_{\sigma\in(\tau)}\sigma
\end{gather*}
\end{proof}
The proof of the theorem gives us a simple algorithm to compute, how the multiplication looks like on our basis elements. We will identify them with the associated conjugacy class in $\mathfrak{S}_{k}$.
\begin{example}
The case $k=1$ is trivial, since there is only one basis element ($1=Id$) with $1.1=1$.
\end{example}
\begin{example}
For $k=2$ we have two basis elements - one is the identity ($1=Id$), the second one (denoted by $x$) can be identified with transposition $(12)$. Now we have
\begin{gather*}
1.1=(Id)(Id)=(Id)=1\qquad x.x=(12)(12)=(Id)=1\\
1.x=(Id)(12)=(12)=x\qquad x.1=(12)(Id)=(12)=x
\end{gather*}
\end{example}
\begin{example}
For $k=3$ we have three basis elements - one is the identity ($1=Id$), the second one (denoted by $x$) can be identified with transposition, and the third one (y) corresponds to the three-cycle. Now we have
\begin{gather*}
1.1=(Id)(Id)=(Id)=1\quad 1.x=(Id)\frac{(12)+(13)+(23)}{3}=\frac{(12)+(13)+(23)}{3}=x \\
1.y=(Id)\frac{(123)+(132)}{2}=\frac{(123)+(132)}{2}=y\\
x.1=\frac{(12)+(13)+(23)}{3}(Id)=\frac{(12)+(13)+(23)}{3}=x\\ x.x=\frac{(12)+(13)+(23)}{3}\frac{(12)+(13)+(23)}{3}=\frac{(Id)+(123)+(132)}{3}=\frac{1+2y}{3}\\
x.y=\frac{(12)+(13)+(23)}{3}\frac{(123)+(132)}{2}=\frac{(12)+(13)+(23)}{3}=x\\
y.1=\frac{(123)+(132)}{2}(Id)=\frac{(123)+(132)}{2}=y\\
y.x=\frac{(123)+(123)}{2}\frac{(12)+(13)+(23)}{3}=\frac{(12)+(13)+(23)}{3}=x\\
y.y=\frac{(123)+(132)}{2}\frac{(123)+(132)}{2}=\frac{2(Id)+(123)+(132)}{4}=\frac{1+y}{2}
\end{gather*}
\end{example}
\begin{corollary}
The algebra $\End_{SL(V)}S^{k}_{0}\mathfrak{sl}(V)$ is commutative.
\end{corollary}
\begin{proof}
For any two permutations $\sigma,\sigma'\in\mathfrak{S}_{k}$, the conjugacy class of $\sigma\sigma'$ is the same as that of $\sigma'\sigma$.
\end{proof}
The Double Commutant Theorem (see \cite{goodwall}) says that the algebra $\End_{SL(V)}S^{k}_{0}\mathfrak{sl}(V)$ (which will be denoted by $\mathcal{A}$ in the sequel) is semisimple, so it has to be a direct sum of commutative simple (and hence one-dimensional) algebras. The number of simple factors is given by $\dim\mathcal{A}$. It is at the same time the number of inequivalent irreducible representations of $\mathcal{A}$, which are one-dimensional. This implies that the $SL(V)$-isotypic components of $S^{k}_{0}\mathfrak{sl}(V)$ are irreducible.
\begin{theorem}
The algebra $\End_{SL(V)}S^{k}_{0}\mathfrak{sl}(V)$ is isomorphic to the centre of $\mathbb{C}[\mathfrak{S}_{k}]$.
\end{theorem}
\begin{proof}
We will use the notation from the above examples and proof. An element $\sum_{\sigma\in\mathfrak{S}_{k}}a_{\sigma}\sigma$ commutes with all of $\mathbb{C}[\mathfrak{S}_{k}]$ if and only if it commutes with all elements of $\mathfrak{S}_{k}$, i.e. if for any $\sigma'\in\mathfrak{S}_{k}$ we have
\begin{displaymath}
\sigma'(\sum_{\sigma\in\mathfrak{S}_{k}}a_{\sigma}\sigma)(\sigma')^{-1}=\sum_{\sigma\in\mathfrak{S}_{k}}a_{\sigma}\sigma
\end{displaymath}
So the coefficients at conjugate permutations should be equal.
\end{proof}

\subsection{Decomposition of $S^{k}_{0}\mathfrak{sl}(V)$}

Now if $s,s'\in\mathfrak{S}_{2k}$ interchanging odd and even numbers, s.t. $\tilde{\sigma}^{s}=\tilde{\sigma}^{s'}$, the action of $C_{s}$ and $C_{s'}$ on \emph{symmetric} tensors is \emph{the same}. This means that it suffices to sum over $s$ with $\sigma^{s}_{1}=Id$, which leads to the following description of $C_{(s)}$: its action on a symmetric tensor is a symmetric tensor, which for given upper index $I$ has symmetries in lower index $J$ given by $\sum_{s\in(s)}s$.
Similarly, it suffices to sum over $s$ with $\sigma^{s}_{2}=Id$, which leads to the following description of $C_{s}$: its action on a symmetric tensor is a symmetric tensor, which for given lower index $J$ has symmetries in upper index $I$ given by $\sum_{s\in(s)}s^{-1}=\sum_{s\in(s)}s$. \par
In the same way it is possible to characterize the action of any element in $\End_{SL(V)}S^{k}\mathfrak{sl}(V)\cong Z(\mathbb{C}[\mathfrak{S}_{k}])$. Concretely, we can imagine the action in two ways. In the first way, the element fixes the upper index and permutes the lower index ($\sigma^{x}_{1}=Id$ and $\sigma^{x}_{2}=x$) and in the second way, the element fixes the lower index and permutes the upper index ($(\sigma^{x}_{1})^{-1}=x$ and $\sigma^{x}_{2}=Id$), respectively. Here, we extend the notation $\tilde{\sigma}$ from elements of $\mathfrak{S}_{2k}$ interchanging odd and even numbers to their conjugacy classes under $\mathfrak{S}^{1}_{k}\times\mathfrak{S}^{2}_{k}$ corresponding to elements of $Z(\mathbb{C}[\mathfrak{S}_{k}])$. This is in particular true for projectors onto $SL(V)$-irreducible components of $S^{k}_{0}\mathfrak{sl}(V)$, so these irreducible components are closed taking transpose, i.e. the corresponding highest weight is preserved by the unique automorphism of the Dynkin diagram of $\mathfrak{sl}(V)$.
\begin{theorem}\label{dec}
The representation $S^{k}_{0}\mathfrak{sl}(V)$ of $SL(V)$ decomposes as direct sum of irreducibles with highest weight $\lambda-w(\lambda)$ for $\lambda\in Par(k)$ (in the case of stable rank), where $w$ is the longest element of the Weyl group of $SL(V)$.
\end{theorem}
\begin{proof}
We know from above that $Z(\mathbb{C}[\mathfrak{S}_{k}])$ is a sum of one-dimensional ideals. These ideals are generated by minimal idempotents ($p_{\lambda}$) labeled by $\lambda\in Par(k)$ (partitions of $k$ are in one-to-one correspondence with conjugacy classes in $\mathfrak{S}_{k}$ and described in \cite{goodwall}. These idempotents are the projectors onto $SL(V)$-isotypic (in this case irreducible) components of $S^{k}_{0}\mathfrak{sl}(V)$. This means that the irreducible component corresponding to $\lambda$ consists of symmetric totally trace-free tensors, which can be written as sum of tensors of the form $\varepsilon^{I}\otimes e_{p_{\lambda}\cdot J}$ (or as sum of tensors of the form $\varepsilon^{p_{\lambda}\cdot I}\otimes e_{J}$). So the highest weight of the irreducible component corresponding to $\lambda$ is $\lambda+\lambda^{*}$, where $\lambda^{*}$ is the highest weight of the dual representation to that with highest weight $\lambda$. These irreducible components may vanish for $n$ not big enough.
\end{proof}
\begin{proposition}
The irreducible components of $S^{k}_{0}\mathfrak{sl}{V}$ corresponding to $\lambda$ do not vanish for any $\lambda$ with $2\depth(\lambda)\leq n$.
\end{proposition}
\begin{proof}
We shall construct the highest weight vector and prove that it is nonzero. Let's put $m:=\depth(\lambda)$ and define
\begin{gather*}
e_{\lambda}:=e_{1}\otimes\dots\otimes e_{1}\otimes e_{2}\otimes\dots\otimes e_{2}\otimes\dots\otimes e_{m}\otimes\dots\otimes e_{m}\\
\varepsilon^{\lambda}:=\varepsilon^{n}\otimes\dots\otimes\varepsilon^{n}\otimes\varepsilon^{n-1}\otimes\dots\otimes\varepsilon^{n-1}\otimes\dots\otimes
\varepsilon^{n-m+1}\otimes\dots\otimes\varepsilon^{n-m+1}	
\end{gather*}
where in $e_{\lambda}$ each $e_{i}$ occurs $\lambda_{i}$-times, and the same for $\varepsilon^{n-i+1}$ in $\varepsilon^{\lambda}$. We know that we are looking for a tensor which for given lower indices has symmetry in upper indices given by $p_{\lambda}$. So let's start with
\begin{displaymath}
	v:=p_{\lambda}(e_{\lambda})\otimes\varepsilon^{\lambda}
\end{displaymath}
This is evidently a highest weight vector of the right weight. Now we shall prove that its symmetrization in columns is nonzero.\par
It suffices to consider symmetrization over permutations preserving $\varepsilon^{\lambda}$. For such a permutation $\pi$ we have
\begin{displaymath}
	(\pi\circ p_{\lambda})(e_{\lambda})\otimes\varepsilon^{\lambda}=(p_{\lambda}\circ\pi)
	(e_{\lambda})\otimes\varepsilon^{\lambda}=p_{\lambda}(e_{\lambda})\otimes\varepsilon^{\lambda}
\end{displaymath}
\end{proof}

\section{Algebra of symmetries}

\subsection{Statement of Theorem}

We have to identify the vector space of symmetries as an associative algebra. To do this, let us first consider the composition $\mathcal{D}_{V}\mathcal{D}_{W}$ in case $V,W\in\mathfrak{sl}(n+1,1)$. The operators $\mathcal{D}_{V}$ and $\mathcal{D}_{W}$ on $M$ are induced by the ambient operators
\begin{displaymath}
\mathfrak{D}_{V}=V^{B}_{A}(x^{A}\partial_{B}-x_{B}\partial^{A})\quad\mbox{and}\quad\mathfrak{D}_{W}=W^{D}_{C}(x^{C}\partial_{D}-x_{D}\partial^{C})
\end{displaymath}
Their composition is, therefore, induced by
\begin{gather}\label{comp}
\mathfrak{D}_{V}\mathfrak{D}_{W}=V^{B}_{A}W^{D}_{C}(x^{A}x^{C}\partial_{B}\partial_{D}-x^{A}x_{D}\partial_{B}\partial^{C}-x_{B}x^{C}\partial^{A}\partial_{D}+x_{B}x_{D}\partial^{A}\partial^{C})+\\
+V^{S}_{A}W^{D}_{S}x^{A}\partial_{D}+V^{B}_{T}W^{T}_{C}x_{B}\partial^{C}\nonumber
\end{gather}
If we write
\begin{gather*}
V^{B}_{A}W^{D}_{C}=T^{BD}_{AC}+\delta^{B}_{C}U^{D}_{A}+\delta^{D}_{A}\tilde{U}^{B}_{C}-\frac{1}{n+2}(\delta^{B}_{A}U^{D}_{C}+\delta^{B}_{A}\tilde{U}^{D}_{C}+\delta^{D}_{C}U^{B}_{A}+\delta^{D}_{C}\tilde{U}^{B}_{A})+\\
+\frac{1}{(n+2)^{2}}\delta^{B}_{A}\delta^{D}_{C}(U^{X}_{X}+\tilde{U}^{X}_{X})
\end{gather*}
where
\begin{gather}\label{u}
U^{D}_{A}=\frac{(2n^{2}+8n+4)V^{X}_{A}W^{D}_{X}+4V^{D}_{X}W^{X}_{A}}{2n(n+2)(n+4)}-\frac{(n^{2}+4n+6)V^{X}_{Y}W^{Y}_{X}\delta^{D}_{A}}{2n(n+1)(n+3)(n+4)}\nonumber\\
\tilde{U}^{B}_{C}=\frac{4V^{X}_{C}W^{B}_{X}+(2n^{2}+8n+4)V^{B}_{X}W^{X}_{C}}{2n(n+2)(n+4)}-\frac{(n^{2}+4n+6)V^{X}_{Y}W^{Y}_{X}\delta^{B}_{C}}{2n(n+1)(n+3)(n+4)}
\end{gather}
then it is easy to verify that $T^{BD}_{AC}$ is totally trace-free. Now, from (\ref{comp}), we may rewrite
\begin{gather*}
\mathcal{D}_{V}\mathcal{D}_{W}=T^{BD}_{AC}(x^{A}x^{C}\partial_{B}\partial_{D}-x^{A}x_{D}\partial_{B}\partial^{C}-x_{B}x^{C}\partial^{A}\partial_{D}+x_{B}x_{D}\partial^{A}\partial^{C})+\\
+U^{D}_{A}(x^{A}x^{C}\partial_{C}\partial_{D}-x^{A}x_{D}\partial_{C}\partial^{C}-x_{C}x^{C}\partial^{A}\partial_{D}+x_{C}x_{D}\partial^{A}\partial^{C})+\\
+\tilde{U}^{B}_{C}(x^{A}x^{C}\partial_{B}\partial_{A}-x^{A}x_{A}\partial_{B}\partial^{C}-x_{B}x^{C}\partial^{A}\partial_{A}+x_{B}x_{A}\partial^{A}\partial^{C})-\\
-\frac{1}{n}(U^{D}_{C}+\tilde{U}^{D}_{C})(x^{A}x^{C}\partial_{A}\partial_{D}-x^{A}x_{D}\partial_{A}\partial^{C}-x_{A}x^{C}\partial^{A}\partial_{D}+x_{A}x_{D}\partial^{A}\partial^{C})-\\
-\frac{1}{n}(U^{B}_{A}+\tilde{U}^{B}_{A})(x^{A}x^{C}\partial_{B}\partial_{C}-x^{A}x_{C}\partial_{B}\partial^{C}-x_{B}x^{C}\partial^{A}\partial_{C}+x_{B}x_{C}\partial^{A}\partial^{C})+\\
+\frac{1}{n^{2}}(U^{X}_{X}+\tilde{U}^{X}_{X})(x^{A}x^{C}\partial_{A}\partial_{C}-x^{A}x_{C}\partial_{A}\partial^{C}-x_{A}x^{C}\partial^{A}\partial_{C}+x_{A}x_{C}\partial^{A}\partial^{C})+\\
+V^{S}_{A}W^{D}_{S}x^{A}\partial_{D}+V^{B}_{T}W^{T}_{C}x_{B}\partial^{C}
\end{gather*}
and, bearing in mind (\ref{u}) and that $x^{C}\partial_{C}$ induces multiplication by $w_{1}$, and $x_{C}\partial^{C}$ induces multiplication by $w_{2}$, respectively, if $f$ has homogeneity $(w_{1},w_{2})$, then
\begin{gather*}
\mathcal{D}_{V}\mathcal{D}_{W}f=T^{BD}_{AC}(x^{A}x^{C}\partial_{B}\partial_{D}-x^{A}x_{D}\partial_{B}\partial^{C}-x_{B}x^{C}\partial^{A}\partial_{D}+x_{B}x_{D}\partial^{A}\partial^{C})f-\\
-r(U^{D}_{A}\partial^{A}\partial_{D}+\tilde{U}^{B}_{C}\partial_{B}\partial^{C})f-(U^{D}_{A}x^{A}x_{D}+\tilde{U}^{B}_{C}x_{B}x^{C})\tilde{\Delta}f+\\
+\frac{n+2}{n(n+4)}(V^{D}_{X}W^{X}_{A}+V^{X}_{A}W^{D}_{X})\frac{nw_{1}+2w_{2}-n}{n+2}x^{A}\partial_{D}f+\\
+V^{X}_{A}W^{D}_{X}x^{A}\partial_{D}f+\\
+\frac{n+2}{n(n+4)}(V^{D}_{X}W^{X}_{A}+V^{X}_{A}W^{D}_{X})\frac{2w_{1}+nw_{2}-n}{n+2}x_{D}\partial^{A}f+\\
+V^{D}_{X}W^{X}_{A}x_{D}\partial^{A}f+\\
+\frac{(w_{1}-w_{2})^{2}-(w_{1}+w_{2})}{(n+1)(n+2)(n+3)}V^{X}_{Y}W^{Y}_{X}f-\\
-\frac{2(n^{2}+4n+6)V^{X}_{Y}W^{Y}_{X}}{n(n+1)(n+2)(n+3)(n+4)}\left[n(w_{1}^{2}+w_{2}^{2}-w_{1}-w_{2})+4w_{1}w_{2}\right]f
\end{gather*}
In particular, if $n+w_{1}+w_{2}=0$, then
\begin{gather*}
\mathcal{D}_{V}\mathcal{D}_{W}f=T^{BD}_{AC}(x^{A}x^{C}\partial_{B}\partial_{D}-x^{A}x_{D}\partial_{B}\partial^{C}-x_{B}x^{C}\partial^{A}\partial_{D}+x_{B}x_{D}\partial^{A}\partial^{C})f-\\
-r(U^{D}_{A}\partial^{A}\partial_{D}+\tilde{U}^{B}_{C}\partial_{B}\partial^{C})f-(U^{D}_{A}x^{A}x_{D}+\tilde{U}^{B}_{C}x_{B}x^{C})\tilde{\Delta}f+\\
+\frac{(n-2)(w_{1}-w_{2})-n(n+4)}{2n(n+4)}V^{D}_{X}W^{X}_{A}x^{A}\partial_{D}f+\\
+\frac{(n-2)(w_{1}-w_{2})+n(n+4)}{2n(n+4)}V^{X}_{A}W^{D}_{X}x^{A}\partial_{D}f+\\
+\frac{(2-n)(w_{1}-w_{2})+n(n+4)}{2n(n+4)}V^{D}_{X}W^{X}_{A}x_{D}\partial^{A}f+\\
+\frac{(2-n)(w_{1}-w_{2})-n(n+4)}{2n(n+4)}V^{X}_{A}W^{D}_{X}x_{D}\partial^{A}f+\\
+\frac{(-n^{3}-n^{2}+6n+12)(w_{1}-w_{2})^{2}-n^{2}(n+4)(n^{2}+4n+5)}{n(n+1)(n+2)(n+3)(n+4)}V^{X}_{Y}W^{Y}_{X}f=\\
=T^{BD}_{AC}(x^{A}x^{C}\partial_{B}\partial_{D}-x^{A}x_{D}\partial_{B}\partial^{C}-x_{B}x^{C}\partial^{A}\partial_{D}+x_{B}x_{D}\partial^{A}\partial^{C})f-\\
-r(U^{D}_{A}\partial^{A}\partial_{D}+\tilde{U}^{B}_{C}\partial_{B}\partial^{C})f-(U^{D}_{A}x^{A}x_{D}+\tilde{U}^{B}_{C}x_{B}x^{C})\tilde{\Delta}f+\\
+\frac{(n-2)(w_{1}-w_{2})(V^{D}_{X}W^{X}_{A}+V^{X}_{A}W^{D}_{X}-\frac{2}{n+2}V^{X}_{Y}W^{Y}_{X}\delta^{A}_{D})}{2n(n+4)}(x^{A}\partial_{D}-x_{D}\partial^{A})f-\\
-\frac{n(n+4)(V^{D}_{X}W^{X}_{A}-V^{X}_{A}W^{D}_{X})}{2n(n+4)}(x^{A}\partial_{D}-x_{D}\partial^{A})f+\\
+\frac{(n^{2}+n+6)(w_{1}-w_{2})^{2}-n^{2}(n+4)(n^{2}+4n+5)}{n(n+1)(n+2)(n+3)(n+4)}V^{X}_{Y}W^{Y}_{X}f
\end{gather*}
Noting that $x^{A}x^{C}\partial_{B}\partial_{D}-x^{A}x_{D}\partial_{B}\partial^{C}-x_{B}x^{C}\partial^{A}\partial_{D}+x_{B}x_{D}\partial^{A}\partial^{C}$ is symmetric in $AB$ and $CD$, we may rewrite the first term and obtain, upon restriction to $\mathcal{N}$,
\begin{gather}\label{compd}
\mathcal{D}_{V}\mathcal{D}_{W}\equiv\mathcal{D}_{(VW)_{2}}+\mathcal{D}_{(VW)_{1}}+\mathcal{D}_{(VW)_{0}}\quad\mod\;\tilde{\Delta}
\end{gather}
where
\begin{eqnarray}
((VW)_{2})^{BD}_{AC}&=&(T^{BD}_{AC}+T^{DB}_{CA})/2\\
((VW)_{1})^{BD}_{AC}&=&\frac{(n-2)(w_{1}-w_{2})(V^{D}_{X}W^{X}_{A}+V^{X}_{A}W^{D}_{X}-\frac{2}{n+2}V^{X}_{Y}W^{Y}_{X}\delta^{A}_{D})}{2n(n+4)}-\nonumber\\
&&-\frac{(V^{D}_{X}W^{X}_{A}-V^{X}_{A}W^{D}_{X})}{2}\nonumber\\
((VW)_{0})^{BD}_{AC}&=&\frac{(n^{2}+n+6)(w_{1}-w_{2})^{2}-n^{2}(n+4)(n^{2}+4n+5)}{n(n+1)(n+2)(n+3)(n+4)}V^{X}_{Y}W^{Y}_{X}\nonumber
\end{eqnarray}
Each of these terms has a simple interpretation as follows. The representation $\otimes^{2}\mathfrak{sl}(n+2,\mathbb{C})$ of $\mathfrak{sl}(n+2,\mathbb{C})$ decomposes into seven irreducible pieces:
\begin{gather}
(\one\boxtimes\one^{*})\otimes(\one\boxtimes\one^{*})=\secsym\boxtimes\secsym^{*}\oplus\secskewsym\boxtimes\secskewsym^{*}\oplus\one\boxtimes\one^{*}\oplus\mathbb{C}\oplus\\
\oplus\secsym\boxtimes\secskewsym^{*}\oplus\secskewsym\boxtimes\secsym^{*}\oplus\one\boxtimes\one^{*}
\end{gather}
The projection of $V\otimes W$ into the first of these irreducibles is $V\boxtimes W$ - their Cartan product. The sum of first two components is the trace-free part of $S^{2}\mathfrak{sl}(n+2,\mathbb{C})$, the sum of the fifth and sixth component is the trace-free part of $\Lambda^{2}\mathfrak{sl}(n+2,\mathbb{C})$. These two components are mapped to zero by $\mathcal{D}$. The projection on the last component is the Lie bracket and the projection on the fourth component is the Killing form.\par
There is another term (in the third tensor power of $\mathfrak{sl}(n+2,\mathbb{C})$), which induces zero symmetry, namely the one of the form $(V^{[B_{1}B_{2}B_{3}]}_{[A_{1}A_{2}A_{3}]})_{\circ}$. Here we use the fact that by (\ref{hos}) any tensor skew-symmetric in three upper or lower indices must induce zero symmetry.
\begin{theorem}
The algebra $\mathcal{A}_{n}$ of symmetries of the sub-Laplacian on $M$ is isomorphic to the tensor algebra
\begin{displaymath}
\bigoplus_{s=0}^{\infty}\bigotimes^{s}\mathfrak{sl}(n+2,\mathbb{C})
\end{displaymath}
modulo the two-sided ideal generated by the elements
\begin{gather}\label{el1}
V^{B}_{A}W^{D}_{C}-(T^{BD}_{AC}+T^{DB}_{CA})/2-\frac{(n-2)(w_{1}-w_{2})(V^{D}_{X}W^{X}_{A}+V^{X}_{A}W^{D}_{X}-\frac{2}{n+2}V^{X}_{Y}W^{Y}_{X}\delta^{A}_{D})}{2n(n+4)}-\\
-\frac{(V^{D}_{X}W^{X}_{A}-V^{X}_{A}W^{D}_{X})}{2}-\frac{(n^{2}+n+6)(w_{1}-w_{2})^{2}-n^{2}(n+4)(n^{2}+4n+5)}{n(n+1)(n+2)(n+3)(n+4)}V^{X}_{Y}W^{Y}_{X}\nonumber
\end{gather}
and
\begin{gather}\label{el2}
(U^{[B}_{[A}V^{D}_{C}W^{F]}_{E]})_{\circ}
\end{gather}
for $U,V,W\in\mathfrak{sl}(n+2,\mathbb{C})$.
\end{theorem}

\subsection{Proof of Theorem}

Define a map from the tensor algebra to $\mathcal{A}_{n}$ by
\begin{displaymath}
V_{1}\otimes V_{2}\otimes\cdots\otimes V_{s}\mapsto\mathcal{D}_{V_{1}}\mathcal{D}_{V_{2}}\cdots\mathcal{D}_{V_{s}}
\end{displaymath}
extended by linearity. From (\ref{compd}) and (\ref{hos}) it follows that the elements (\ref{el1}) and (\ref{el2}) are mapped to zero. To complete the proof, it suffices to consider the corresponding graded algebras. We must show that the kernel of the mapping
\begin{gather}
\bigoplus_{s=0}^{\infty}\bigotimes^{s}\mathfrak{sl}(n+2,\mathbb{C})\rightarrow\bigoplus_{\depth\lambda\leq2}\lambda\boxtimes\lambda^{*}
\end{gather}
is the two-sided ideal generated by (\ref{el1}) and (\ref{el2}).\par
%\begin{lemma}\label{rozklad} 
%Let $V$ be a finite-dimensional $\mathfrak{sl}(n+2,\mathbb{C})$-module and $A,B,C,D\leq V$ such that $V=A\oplus B=C\oplus D$. Then $V=(A\cap C)\oplus(B+D)$.
%\end{lemma}
%\begin{proof}
%The proof is the same as in \cite{symmetries}. Every (complex) $\mathfrak{sl}(n+2,\mathbb{C})$-module is at the same time a $\mathfrak{sl}(n+2)$-module, so it admits an invariant scalar product, with respect to which are the above decompositions orthogonal. Let $P$ be the projection onto $A$ and $Q$ the projection onto $C$, respectively. Then we want to show that
%\begin{displaymath}
%	V=(\im P\cap\im Q)\oplus(\ker P+\ker Q)
%\end{displaymath}
%This is a fact concerning orthogonal projections. The composition $QP$ preserves $(\im P\cap\im Q)^{\bot}$ and is norm-decreasing there. Hence $Id-QP$ is invertible on this subspace. Therefore, we may write
%\begin{eqnarray*}
%T&=&(Id-QP)(Id-QP)^{-1}T\\
%&=&((Id-P)+(Id-Q)P)(Id-QP)^{-1}T
%\end{eqnarray*}
%for $T\in(\im P\cap\im Q)^{\bot}$, an expression evidently in $\ker P+\ker Q$. On the other hand, $(\ker P+\ker Q)\subset(\im P\cap\im Q)^{\bot}$, since $\ker P\bot\im P$ and $\ker Q\bot\im Q$, respectively.
%\end{proof}
%Similar statement holds for more then one possible way of decomposition by induction.\par
We look at the structure of symmetry algebra. We know that all symmetries are generated by those of first order, which form a representation of $\mathfrak{sl}(n+2,\mathbb{C})$ isomorphic to the adjoint one. So as an algebra, it is a quotient of the tensor algebra of $\mathfrak{sl}(n+2,\mathbb{C})$. Since the result has a structure of $\mathfrak{sl}(n+2,\mathbb{C})$-module, the ideal we factorize must be an invariant subspace. For brevity, we will write $\mathfrak{g}$ instead of $\mathfrak{sl}(n+2,\mathbb{C})$.\par
The smallest tensor power containing a nonempty subset of the ideal is the second one. The homogeneity two part $I_{2}$ of the ideal contains all of the second tensor power except the symmetric trace-free part. We divide it into two parts: $I_{2,1}$ is the alternating power and $I_{2,2}$ is the trace of symmetric power, respectively.
\begin{lemma}
Consider the tensor algebra $\bigotimes\mathfrak{g}$. The two-sided ideal generated by $I_{2,1}$ is a complement of $S(\mathfrak{g})$.
\end{lemma}
\begin{proof}
$I_{2,1}$ is simply $\Lambda^{2}\mathfrak{g}$. Since $S(\mathfrak{g})$ is by definition a quotient of $\otimes\mathfrak{g}$ modulo ideal generated by $\Lambda^{2}\mathfrak{g}$, the ideal must be a vector space complement of $S(\mathfrak{g})$ in $\otimes\mathfrak{g}$.
%We shall proceed by induction. The claim evidently holds in the zeroth, first and second power, respectively. Assume that it holds for $k$-th tensor power. Then by Lemma \ref{rozklad} ($V=\otimes^{k+1}\mathfrak{g}$) the complement of the ideal in the $(k+1)$-st power is $S^{k}\mathfrak{g}\otimes\mathfrak{g}\cap\mathfrak{g}\otimes S^{k}\mathfrak{g}=S^{k+1}\mathfrak{g}$. 
\end{proof}
\begin{lemma}
Consider the algebra $S(\mathfrak{g})$. The two-sided ideal generated by $I_{2,2}$ is a complement of $S^{k}_{0}\mathfrak{g}$.
\end{lemma}
\begin{proof}
$I_{2,2}$ consists of all trace terms in $S^{2}\mathfrak{g}$. The ideal generated by $I_{2,2}$ obviously consists of all tensors with zero trace-free part. Therefore, its vector space complement consists of totally trace-free tensors.
%We shall proceed by induction. The claim evidently holds in the zeroth, first and second power, respectively. Assume that it holds for the $k$-th symmetric power. Then by Lemma \ref{rozklad} ($V=S^{k+1}\mathfrak{g}$) the complement of the ideal in the $(k+1)$-st power is $S^{k}_{0}\mathfrak{g}\odot\mathfrak{g}\cap\mathfrak{g}\odot S^{k}_{0}\mathfrak{g}=S^{k+1}_{0}\mathfrak{g}$.
\end{proof}
Let's remark that the two-sided ideal in $S(\mathfrak{g})$ generated by $I_{2,2}$ is an intersection of $S(\mathfrak{g})$ with the two-sided ideal in $\bigotimes\mathfrak{g}$ generated by $I_{2,2}$.\par
So we see that $I_{2}$ generates in higher homogeneities the complement of the symmetric trace-free power of $\mathfrak{g}$. The trace-free part of the third symmetric power of $\mathfrak{g}$ decomposes by Theorem \ref{dec} as
\begin{displaymath}
\thirdsym\boxtimes\thirdsym^{*}\oplus\corner\boxtimes\corner^{*}\oplus\thirdskewsym\boxtimes\thirdskewsym^{*}
\end{displaymath}
The third terms induces zero symmetries, so it must lie in the ideal. We denote it by $I_{3}$.
\begin{lemma}
Let $T\in\bigotimes^{k}\mathbb{C}^{n+2}$ and let $\lambda\in Par(k)$. Then $p_{\lambda}(T)$ is a sum of tensors skew-symmetric in $\depth(\lambda)$ indices.
\end{lemma}
\begin{proof}
First, $p_{\lambda}=\sum_{A\in STab(\lambda)}p_{A}$, so it suffices to prove the claim for each $p_{A}$. But $p_{A}=Kc(A)r(A)$. The constant $K$ only depends on $\lambda$, so we may consider $K=1$. We have
\begin{gather*}
p_{A}(T)=c(A)r(A)(T)=\sum_{s\in Row(A)}c(A)s(T)=:\sum_{s\in Row(A)}c(A)(T^{s})
\end{gather*}
If $i_{1},\dots,i_{l}$, where $l=\depth(\lambda)$, are entries in the first column of $A$, then $c(A)$ (among others) skew-symmetrizes $T^{s}$ in the $i_{1}$-th, ..., $i_{l}$-th indices.
\end{proof}
\begin{lemma}\label{skews}
Let $T\in\bigotimes^{k}\mathbb{C}^{n+2}$ and let $\lambda\in Par(k)$. Then skew-symmetrizing $p_{\lambda}(T)$ in $\depth(\lambda)+1$ indices gives zero.
\end{lemma}
\begin{proof}
Assume for simplicity, as above, that $p_{\lambda}=\sum_{A\in STab(\lambda)}p_{A}$, where $p_{A}=c(A)r(A)$. Given three indices, let us denote the skew-symmetrizer in these indices as $C$. Then 
\begin{gather*}
Cc_{A}=\sum_{\pi\in Col(A)}(\sgn\pi) C\pi=\sum_{\pi\in Col(A)}(\sgn\pi)\pi C_{\pi}
\end{gather*}
where $C_{\pi}$ is skew-symmetrizer in some (possibly not the same) three indices 
and we have
\begin{displaymath}
Cp_{A}(T)=Cc(A)r(A)(T)=\sum_{\pi\in Col(A)}(\sgn\pi)\pi C_{\pi}r(A)T
\end{displaymath}
Since $C_{\pi}$ skew-symmetrizes in more then $\depth(\lambda)$ indices, the composition $C_{\pi}r(A)$ equals zero, since it at the same time symmetrizes and skew-symmetrizes in the same pair of indices.
\end{proof}
Now the theorem is evidently true up to homogeneity three. In homogeneity $\geq$three the complement of the ideal is given by those totally trace-free symmetric tensors, which, for arbitrary but fixed upper indices, skew-symmetrized in any lower three indices are zero (Lemma \ref{skews}). But totally trace-free symmetric tensors generated by $I_{3}$ are exactly those totally trace-free symmetric tensors, which can be written as a sum of tensors such that their symmetry in upper indices for fixed lower indices is given by $p_{\lambda}$ for some $\lambda$ with $\depth(\lambda)\geq3$, so we see we are done.

\section{Acknowledgements}

This work was supported by grant GACR 201/08/0397.

\bibliographystyle{plain}
\bibliography{zvlasak}

\begin{thebibliography}{10}

\bibitem{sep}
C.P. Boyer, E.G. Kalnins, and W.jun. Miller.
\newblock {Symmetry and separation of variables for the Helmholtz and Laplace
  equations.}
\newblock {\em Nagoya Math. J.}, 60:35--80, 1976.

\bibitem{symlap}
Michael Eastwood.
\newblock {Higher symmetries of the Laplacian.}
\newblock {\em {Annals of Mathematics}}.

\bibitem{goodwall}
Roe Goodman and Nolan~R. Wallach.
\newblock {\em {Representations and invariants of the classical groups.
  Paperback ed.}}
\newblock {Encyclopedia of Mathematics and Its Applications. 68. Cambridge:
  Cambridge University Press. xvi, 685 p. }, 1999.

\bibitem{powsublap}
A.Rod Gover and C.Robin Graham.
\newblock {CR invariant powers of the sub-Laplacian.}
\newblock {\em J. Reine Angew. Math.}, 583:1--27, 2005.

\bibitem{powlap}
C.Robin Graham, Ralph Jenne, Lionel~J. Mason, and George~A.J. Sparling.
\newblock {Conformally invariant powers of the Laplacian. I: Existence.}
\newblock {\em J. Lond. Math. Soc., II. Ser.}, 46(3):557--565, 1992.

\bibitem{sublap}
David Jerison and John~M. Lee.
\newblock {A subelliptic, nonlinear eigenvalue problem and scalar curvature on
  CR manifolds.}
\newblock {Microlocal analysis, Proc. Conf., Boulder/Colo. 1983, Contemp. Math.
  27, 57-63 (1984).}, 1984.

\bibitem{miller}
Willard~jun. Miller.
\newblock {\em {Symmetry and separation of variables. With a foreword by
  Richard Askey.}}
\newblock {Encyclopedia of Mathematics and its Applications. Vol. 4. Reading,
  Massachusetts: Addison-Wesley Publishing Company. XXX, 285 p. \$ 21.50 },
  1977.

\bibitem{diplomka}
V.~Tu\v~cek.
\newblock {Tractor calculi for parabolic geometries}.
\newblock {\em Faculty of Mathematics and Physics, Charles University}, 2008.

\bibitem{Weylstr}
Andreas \v~Cap and Jan Slov\'ak.
\newblock {Weyl structures for parabolic geometries.}
\newblock {\em Math. Scand.}, 93(1):53--90, 2003.

\bibitem{parabook}
Andreas \v~Cap and Jan Slov\'ak.
\newblock {\em {Parabolic geometries I. Background and general theory.}}
\newblock {Mathematical Surveys and Monographs 154. Providence, RI: American
  Mathematical Society (AMS). x, 628~p. \$~120.00 }, 2009.

\bibitem{commsym}
P.~Winternitz and I.~Fri\v~s.
\newblock {Invariant expansions of relativistic amplitudes and subgroups of
  proper Lorentz group}.
\newblock {\em Soviet Journal of Nuclear Physics}, 1:636--\&, 1965.

\end{thebibliography}

\end{document}